\documentclass[oneside,preprint]{amsart}

\usepackage{verbatim, upref, amsmath, amsthm, amsxtra, amssymb, graphicx}

\usepackage{mathtools}

\usepackage{epsfig,color}

\usepackage{hyperref}

\usepackage{varioref}

\usepackage{url}

\DeclareMathAlphabet\mathscr{U}{eus}{m}{n}
\SetMathAlphabet\mathscr{bold}{U}{eus}{b}{n}
\DeclareMathAlphabet\matheur{U}{eur}{m}{n}
\SetMathAlphabet\matheur{bold}{U}{eur}{b}{n}

\numberwithin{equation}{section}

\newtheorem{theo}{Theorem}[section]
\newtheorem{prop}[theo]{Proposition}
\newtheorem{lemm}[theo]{Lemma}
\newtheorem{coro}[theo]{Corollary}
\newtheorem{conj}[theo]{Conjecture}

\theoremstyle{definition}

\newtheorem{defi}[theo]{Definition}
\newtheorem{exam}[theo]{Example}
\newtheorem{exas}[theo]{Examples}

\theoremstyle{remark}

\newtheorem{rema}[theo]{Remark}

\newtheorem{prob}[theo]{Problem}

\newcommand{\dT}{\ensuremath|\hspace{-2pt}|}

\begin{document}
\allowdisplaybreaks\frenchspacing

\title{Representations of toral automorphisms}

\author{Klaus Schmidt}

\address{Mathematics Institute, University of Vienna, Oskar-Morgenstern-Platz 1, A-1090 Vienna, Austria, \textup{and} Erwin Schr\"odinger Institute for Mathematical Physics, Boltzmanngasse~9, A-1090 Vienna, Austria} \email{klaus.schmidt@univie.ac.at}

%\thanks{}



%\dedicatory{}

%\date{}

	\begin{abstract}
This survey gives an account of an algebraic construction of symbolic covers and representations of ergodic automorphisms of compact abelian groups, initiated by A.M. Vershik around 1992 for hyperbolic automorphisms of finite-dimensional tori. The key ingredient in this approach, which was subsequently extended to arbitrary expansive automorphisms of compact abelian groups, is the use of homoclinic points of the automorphism.

Although existence and abundance of homoclinic points is intimately connected to expansiveness of the automorphism, it is nevertheless possible to extend certain aspects of this construction to nonexpansive irreducible automorphisms of compact abelian groups (like irreducible toral automorphisms whose dominant eigenvalue is a Salem number). The later sections of this survey discuss the phenomena and problems arising in this extension.
	\end{abstract}

\maketitle

\section{Introduction: Symbolic covers}

A \textit{topological dynamical system} $(X,T)$ is a pair consisting of a compact metrizable space $X$ and a homeomorphism $T$ of $X$. Two such systems $(X,T)$ and $(X',T')$ are \textit{isomorphic} (or \textit{conjugate}) if there exists an equivariant homeo\-morphism $\phi \colon X\longrightarrow X'$.\footnote{\,\label{equivariant}A map $\phi \colon X\longrightarrow X'$ is \textit{equivariant} (or, more precisely, \textit{$(T,T')$-equivariant}) if $\phi \circ T=T'\circ \phi $.} If there exists a continuous, \textit{surjective}, equivariant map $\psi \colon X\longrightarrow X'$ we say that $(X,T)_\psi $ is a \textit{cover} of $(X',T')$ with \textit{covering map} $\psi $ and call $(X',T')$ a \textit{factor} of $(X,T)$ with \textit{factor map} $\psi $.

If the topological entropies of $(X,T)$ and $(X',T')$ coincide, then $(X,T)_\psi $ is an \textit{equal entropy} cover of $(X,T)$ (this property is, of course, independent of the specific covering map $\psi $). If the covering map $\psi \colon X\longrightarrow X'$ is finite-to-one everywhere (resp. bounded-to-one) then $(X,T)_\psi $ is a \textit{finite-to-one} (resp. \textit{bounded-to-one}) cover of $(X',T')$.  Finally, if $(X',T')$ is topologically transitive and $|\psi ^{-1}(\{x\})|=1$ for every doubly transitive point\footnote{\,\label{transitive}A point $x\in X$ is \textit{doubly transitive} if both the forward and backward semi-orbits of $x$ under $T$ are dense in $X$.} $x\in X'$ we say that $(X,T)_\psi $ is an \textit{almost one-to-one cover} of $(X,T)$.\footnote{\,\label{almost 1-1}There are several different definitions of \textit{almost one-to-one covers}, but this one will do for our purposes. There are also different notions of \textit{topological transitivity}; here we could take `density of doubly transitive points' as an appropriate definition.}

A topological dynamical system $(X,T)$ is \textit{expansive} if
	\begin{equation}
	\label{eq:expansive}
\inf\nolimits_{x,x'\in X:\,x\ne x'} \;\sup\nolimits_{n\in \mathbb{Z}}\;d(T^nx,T^nx')>0
	\end{equation}
for some (and hence for every) metric $d$ which induces the topology of $X$. If $X$ is zero-dimensional and $(X,T)$ is expansive we call $(X,T)$ a \textit{symbolic system}.

If $(X,T)$ is a symbolic system it is isomorphic to a \textit{shift space} $(\Omega ,\sigma )$, where $\mathsf{A}$ is a finite set (called an \textit{alphabet}),
	\begin{equation}
	\label{eq:shift}
(\sigma \omega )_n=\omega _{n+1},\enspace n\in \mathbb{Z},
	\end{equation}
is the \textit{shift} on $\mathsf{A}^\mathbb{Z}$, and $\Omega \subset \mathsf{A}^\mathbb{Z}$ is a closed, shift-invariant set. Here we do not distinguish notationally between $\sigma $ and its restriction $\sigma |_\Omega $ to $\Omega $.

Recall that a shift space $\Omega \subset \mathsf{A}^\mathbb{Z}$ is \textit{of finite type} (abbreviated as \textit{SFT}) if there exists a finite subset $F\subset \mathbb{Z}$ such that
	\begin{equation}
	\label{eq:finite type}
\Omega =\{\omega \in \mathsf{A}^\mathbb{Z}:\pi _F(\sigma ^n\omega )\in \pi _F(\Omega )\;\textup{for every}\;n\in \mathbb{Z}\},
	\end{equation}
where $\pi _F$ is the projection of each $\omega \in \mathsf{A}^\mathbb{Z}$ onto its coordinates in $F$. A symbolic system $(X,T)$ is \textit{of finite type} if it is isomorphic to a shift space $(\Omega ,\sigma )$ for some finite alphabet $\mathsf{A}$ and some \textit{SFT} $\Omega \subset \mathsf{A}^\mathbb{Z}$. A more intrinsic definition of symbolic systems of finite type can be given in terms of a descending chain condition: a symbolic system $(X,T)$ is of finite type if and only if every sequence $(X_n,T_n)_{n\ge1}$ of symbolic systems with $X_{n+1}\subset X_n$ and $T_{n+1}=T_n|_{X_{n+1}}$ for every $n\ge1$, and with $X=\bigcap_{n\ge 1}X_n$, satisfies that $X=X_N$ for some $N\ge1$.

A symbolic system $(X,T)$ is \textit{sofic} if it is a factor of a symbolic system of finite type.

An almost one-to-one symbolic cover $(X,T)_\psi $ of a topological dynamical system $(X',T')$ is a \textit{symbolic representation} of $(X',T')$.

\smallskip Representations of smooth dynamical systems (like hyperbolic toral automorphisms) by symbolic systems which are sofic or of finite type are extremely useful for determining dynamical properties of the systems which would be much more difficult to obtain by other means. There are many classical examples of such representations, most importantly the ones arising from \textit{Markov partitions} of hyperbolic toral automorphisms and, more generally, of axiom A diffeomorphisms, described in the papers by Adler-Weiss \cite{AW, AW2}, Sinai \cite{Sinai1, Sinai2} and Bowen \cite{Bowen1} at varying levels of generality.

Let me recall the notion of a Markov partition in a particularly simple example: the automorphism $\alpha _A$ of $\mathbb{T}^2=\mathbb{R}^2/\mathbb{Z}^2$ defined by the hyperbolic matrix
	\begin{equation}
	\label{eq:A}
A=
	\left(
	\begin{matrix}
0&1
	\\
1&1
	\end{matrix}\right)
\in \textup{GL}(2,\mathbb{Z}).
	\end{equation}
The matrix $A$ has one-dimensional expanding and contracting eigenspaces $v_+$ and $v_-$, respectively. Under the quotient map $\pi \colon \mathbb{R}^2\longrightarrow \mathbb{T}^2$ these eigenspaces get sent to dense $\alpha _A$-invariant subgroups $\pi (v_{\pm})$ of $\mathbb{T}^2$ which intersect in a countable dense subgroup $\Delta _{\alpha _A}(\mathbb{T}^2)= \pi (v_+)\cap \pi (v_-)\subset \mathbb{T}^2$. Every $w\in \Delta _{\alpha _A}(\mathbb{T}^2)$ is \textit{homoclinic to $0$}, or simply \textit{homoclinic}, in the sense that
	\begin{equation}
	\label{eq:homoclinic}
\lim_{|n|\to\infty }\alpha _A^nw=0.
	\end{equation}
In Figure 1 four of these homoclinic points are marked with the symbols $x^\Delta ,y^\Delta ,\linebreak[0]z^\Delta $ and $w^\Delta $.

	\begin{figure}[ht]
\vspace{-5mm}\hspace{-2mm}\includegraphics[width=70mm]{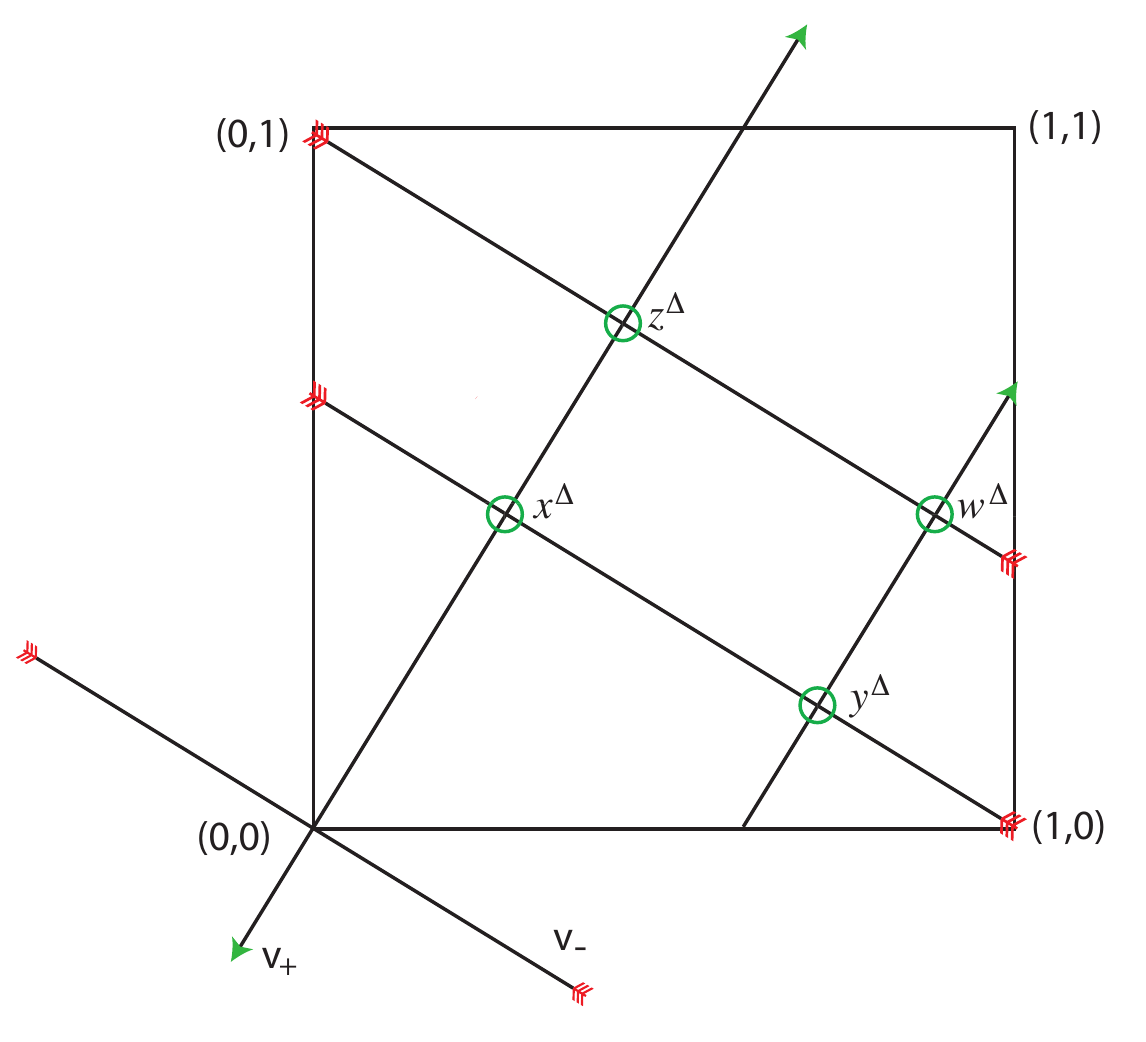}\hspace{6mm}\raisebox{11.3mm}{\includegraphics[width=46.5mm]{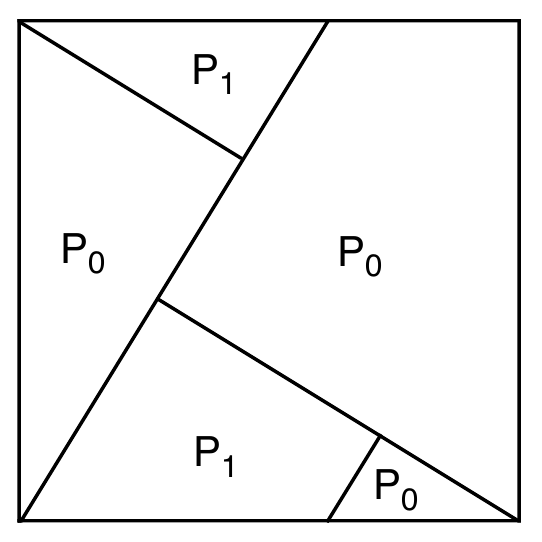}}\vspace{-14mm}
	\\
\begin{center}\textsc{\small ~\hspace{11mm}Figure~1\label{Figure1}} \hspace{45.5mm}\textsc{\small Figure~2\label{Figure2}}\end{center}\vspace{3mm}
	\end{figure}

Figure 2 shows a cover of $\mathbb{T}^2$ by two closed rectangles $\mathsf{P}_0$ and $\mathsf{P}_1$ with disjoint interiors, whose vertices are homoclinic points, and whose edges are certain connected subsets of the dense subgroups $\pi (v_\pm)\subset \mathbb{T}^2$ drawn in Figure 1. One easily checks that $\alpha _A(\mathsf{P}_1)\subset \mathsf{P}_0$ and $\alpha _A(\mathsf{P}_0)\subset \mathsf{P}_0 \cup \mathsf{P}_1$. Furthermore, if we associate with every $x\in \mathbb{T}^2$ the sequence $\psi (x)=(\psi (x)_n)_{n\in \mathbb{Z}}\in \{0,1\}^\mathbb{Z}$ with
	\begin{displaymath}
\psi (x)_n =
	\begin{cases}
1 &\textup{if}\;\alpha _A^nx\in \mathsf{P}_1,
	\\
0&\textup{otherwise},
	\end{cases}
	\end{displaymath}
we obtain an equivariant Borel map $\psi $ from $\mathbb{T}^2$ to the `golden mean shift'
	\begin{equation}
	\label{eq:golden}
\Omega _{GM}=\{\omega =(\omega _k)_{k\in \mathbb{Z}}\in \{0,1\}^\mathbb{Z}: \omega _k\omega _{k+1}=0\;\textup{for every}\;k\in \mathbb{Z}\},
	\end{equation}
which admits a continuous, surjective, equivariant, and at most two-to-one map $\phi \colon \Omega _{GM}\linebreak[0]\longrightarrow \mathbb{T}^2$ satisfying $\phi \circ \psi (x)=x$ for every $x\in \mathbb{T}^2$.
	\label{cover1}
Then $(\Omega _{GM},\sigma )_\phi $ is a symbolic representation of finite type of $(\mathbb{T}^2,\alpha _A)$, and the cover $\mathsf{P}=\{\mathsf{P}_0,\mathsf{P}_1\}$ is called a \textit{Markov partition} of $(\mathbb{T}^2,\alpha _A)$ (although it is not, of course, a partition). A detailed discussion of this construction requires a bit of care (cf., e.g., \cite{Adler}).

The geometric construction of Markov partitions for general irreducible\footnote{\,An automorphism $\alpha $ of a compact abelian group $X$ is \textit{irreducible} if every closed, $\alpha $-invariant subgroup $Y\subsetneq X$ is finite.} hyperbolic automorphisms of $\mathbb{T}^n,\,n\ge3$, yields much more complicated sets whose boundaries cannot be smooth (cf. \cite{Bowen 2}). For a nice overview of the quite intricate geometric constructions of Markov partitions for Pisot- and more general hyperbolic automorphisms of $\mathbb{T}^n$ we refer to \cite{AHFI}.

A rather different approach to symbolic representations of toral automorphisms has its origins in the paper \cite{Ver2} by Vershik, who obtained a representation of the toral automorphism $\alpha _A$ in \eqref{eq:A} by the golden mean shift $\Omega _{GM}$ in \eqref{eq:golden} by using homoclinic points rather than Markov partitions. In a series of papers this construction was subsequently extended to arbitrary hyperbolic toral automorphisms (cf., e.g., \cite{SV, S3, Sidorov}); related, but somewhat different, constructions appear in \cite{Borgne, KV}.

In \cite{ES}, a systematic approach to Vershik's `homoclinic' construction of symbolic covers of expansive group automorphisms (and, more generally, of expansive $\mathbb{Z}^d$-actions by automorphisms of compact abelian groups) was developed, based on the analysis of the homoclinic group of expansive algebraic $\mathbb{Z}^d$-actions in \cite{LS}. In all these considerations, the hypothesis of expansiveness is (almost) indispensable. If the condition of expansiveness is weakened, there may be no nonzero homoclinic points and most of the machinery described here is either unavailable or has to be modified considerably (cf. Section \ref{s:nonexpansive}).

\smallskip Let me briefly describe Vershik's approach in the case of the already familiar toral automorphism $\alpha _A$ in \eqref{eq:A}.

We write $\Delta _{\alpha _A}(\mathbb{T}^2)=\pi (v_+)\cap \pi (v_-)$ for the homoclinic group of $\alpha _A$ and take a nonzero point $w\in \Delta _{\alpha _A}(\mathbb{T}^2)$. Since the convergence in \eqref{eq:homoclinic} is exponentially fast as $|n|\to\infty $, we obtain a well-defined group homomorphism $\xi _w\colon \ell ^\infty (\mathbb{Z},\mathbb{Z})\longrightarrow \mathbb{T}^2$ (where $\ell ^\infty (\mathbb{Z},\mathbb{Z})$ is the group of bounded two-sided integer sequences with coordinate-wise addition) by setting
	\begin{equation}
	\label{eq:xi_w}
\xi _w(v) = \sum\nolimits_{n\in \mathbb{Z}}v_n\alpha _A^{-n}w
	\end{equation}
for every $v=(v_n)_{n\in \mathbb{Z}}\in \ell ^\infty (\mathbb{Z},\mathbb{Z})$. The homomorphism $\xi _w$ is clearly $(\alpha _A,\bar{\sigma })$-equivariant (cf. Footnote \ref{equivariant} \vpageref{equivariant}), where $\bar{\sigma }$ is the shift $(\bar{\sigma }v)_n=v_{n+1}$ on $\ell ^\infty (\mathbb{Z},\mathbb{Z})$. In \cite{Ver3} Vershik showed that the restriction of $\xi _w$ to the golden mean shift $\Omega _{GM}\subset \ell ^\infty (\mathbb{Z},\mathbb{Z})$ in \eqref{eq:golden} is a surjective map from $\Omega _{GM}$ to $\mathbb{T}^2$, and that $(\Omega _{GM},\bar{\sigma })_{\xi _w}$ is a bounded-to-one symbolic cover of $(\mathbb{T}^2,\alpha _A)$. If $w\in\Delta _{\alpha _A}(\mathbb{T}^2)$ a `good' homoclinic point (like $x^\Delta ,y^\Delta ,z^\Delta $, but not $w^\Delta $, in Figure 1 \vpageref{Figure1}), the covering map $\xi _w$ is almost one-to-one, so that $(\Omega _{GM},\bar{\sigma })_{\xi _w}$ becomes a symbolic representation of $(\mathbb{T}^2,\alpha _A)$ (cf. \cite{SV}).

Note that the only difference between the representations $(\Omega _{GM},\sigma )_\phi $ above and $(\Omega _{GM},\bar{\sigma })_{\xi _{x^\Delta }}$ lies in the choice of the covering maps.\footnote{\,The notational distinction between the shift operator $\sigma $ on $\Omega _{GM}$ and the restriction of $\bar{\sigma }$ to $\Omega _{GM}\subset \ell ^\infty (\mathbb{Z},\mathbb{Z})$ is rather pedantic (since they coincide); it is intended to remind the reader that there is nothing really special about the \textit{SFT} $\Omega _{GM}\subset \ell ^\infty (\mathbb{Z},\mathbb{Z})$, and that there are other closed, bounded, $\bar{\sigma }$-invariant subsets $\Omega '\subset \ell ^\infty (\mathbb{Z},\mathbb{Z})$ which could serve equally well as symbolic representations of $(\mathbb{T}^2,\alpha _A)$ with covering map $\xi _{x^\Delta }\negthinspace|_{\Omega '}$ (cf., e.g., Corollary \ref{c:1-1-cover}).}

\smallskip In order to describe in greater detail the homoclinic construction of symbolic covers and representations for expansive automorphisms of compact connected abelian groups we have to discuss homoclinic points of expansive group automorphisms, with a little excursion into the nonexpansive case.

\section{Irreducible automorphisms of compact abelian groups}\label{s:irreducible}
	\begin{defi}
	\label{d:homoclinic}
Let $\alpha $ be a continuous automorphism of a compact abelian group $X$ with identity element $0=0_X$. A point $x \in X$ is \emph{$\alpha $-homoclinic} (or simply \emph{homoclinic}) if $\lim_{|n|\to \infty }\alpha ^ nx=0$. The set $\Delta _\alpha (X)$ of homoclinic points in $X$ is an $\alpha $-invariant subgroup.
	\end{defi}

Recall that two continuous automorphisms $\alpha $ and $\beta $ of compact abelian groups $X$ and $Y$ are \emph{finitely equivalent} if there exist continuous, surjective, equivariant and finite-to-one group homomorphisms $\phi \colon X \longrightarrow Y$ and $\psi \colon Y \longrightarrow X$. In order to describe all irreducible automorphisms of compact abelian groups up to finite equivalence we denote by $R_1=\mathbb{Z}[u ^{\pm 1}]$ the ring of Laurent polynomials with integral coefficients and write every $h \in R_1$ as
	\begin{equation}
	\label{eq:h}
h=\sum\nolimits_{k \in \mathbb{Z}}h_ku ^ k
	\end{equation}
with $h_k \in \mathbb{Z}$ for all $k$ and $h_k=0$ for all but finitely many $k$. Fix an irreducible polynomial
	\begin{equation}
	\label{eq:f}
f=f_0+\dots +f_mu ^ m \in R_1
	\end{equation}
with $m>0$, $f_m>0$ and $f_0\ne0$, denote by $\Theta _f$ the set of roots of $f$, and set
	\begin{equation}
	\label{eq:Omega}
	\begin{gathered}
\Theta _f ^-=\{ \theta \in \Theta _f:|\theta |<1 \},\enspace \enspace \Theta _f ^+=\{ \theta \in \Theta _f:|\theta |>1 \},
	\\
\enspace \Theta _f ^\circ =\{ \theta \in \Theta _f:|\theta |=1 \}.
	\end{gathered}
	\end{equation}
We define the shift $\sigma \colon \mathbb{T}^ \mathbb{Z}\longrightarrow \mathbb{T}^ \mathbb{Z}$ as in \eqref{eq:shift} by
	\begin{displaymath}
\sigma (x)_n=x_{n+1}
	\end{displaymath}
for every $x=(x_n)\in \mathbb{T}^ \mathbb{Z}$ and consider, for every nonzero $h \in R_1$ of the form \eqref{eq:h}, the shift-commuting surjective group homomorphism
	\begin{equation}
	\label{eq:hsigma}
\smash{h(\sigma )=\sum\nolimits_{k \in \mathbb{Z}}h_k \sigma ^ k\colon \mathbb{T}^ \mathbb{Z}\longrightarrow \mathbb{T}^ \mathbb{Z}.}
	\end{equation}
Put
	\begin{equation}
	\label{eq:principal}
X_f=\{ x \in \mathbb{T}^{\mathbb{Z}}: f(\sigma )(x)=0 \} = \ker f(\sigma ),
	\end{equation}
and write
	\begin{equation}
	\label{eq:alpha2}
\alpha _f= \sigma |_{X_f}
	\end{equation}
for the restriction of $\sigma $ to $X_f \subset \mathbb{T}^{\mathbb{Z}}$. By \cite[Theorem 7.1 and Propositions 7.2 -- 7.3]{DSAO}, $\alpha _f$ is nonexpansive if and only if $\Theta _f ^\circ \ne\varnothing $, and ergodic if and only if $f$ is not cyclotomic (i.e. if and only if $f$ does not divide $u ^ m-1$ for any $m\ge1$). In view of this we adopt the following terminology.
	\begin{defi}
	\label{d:hyperbolic}
The polynomial $f$ in \eqref{eq:f} is \emph{hyperbolic} if $\Theta _f ^\circ =\varnothing $, \emph{nonhyperbolic} if $\Theta _f ^\circ \ne\varnothing $, and \emph{cyclotomic} if $\Theta _f ^\circ$ contains a root of unity.
	\end{defi}

According to \cite{S2}, every irreducible automorphism $\alpha $ of a compact abelian group $X$ is finitely equivalent to an automorphism of the form $\alpha _f$ for some irreducible polynomial $f \in R_1$.

	\begin{exam}
	\label{e:companion}
If the polynomial $f$ in \eqref{eq:f} satisfies that $f_m=|f_0|=1$, then $X_f$ is isomorphic to $\mathbb{T}^ m=\mathbb{R}^ m/\mathbb{Z}^ m$, and the shift $\alpha _f$ is conjugate to the companion matrix
	\begin{equation}
	\label{eq:companion}
M_f= \left[
	\begin{smallmatrix}
0&1&0&\dots&0&0
	\\
0&0&1&\dots&0&0
	\\
\vdots&&\vdots&\ddots&\vdots&0
	\\
0&0&0&\dots&0&1
	\\
-f_0&-f_1&-f_2&\dots&-f_{m-2}&-f_{m-1}
	\end{smallmatrix}
\right],
	\end{equation}
acting on $\mathbb{T}^{m}$ from the left, where the isomorphism between $X_f$ and $\mathbb{T}^{m}$ is the coordinate projection\vspace{-2mm}
	$$
x \mapsto \left[
	\begin{smallmatrix}
x_0
	\\
x_1
	\\
\vdots
	\\
x_{m-1}
	\end{smallmatrix}
\right].
	$$
	\end{exam}

	\begin{exas}
	\label{e:x2}
Consider the irreducible polynomials $f_1=2,\enspace f_2=2-u,\enspace f_3=3-2u$ in $R_1$. Then $\alpha _{f_i}$ is the shift on $X_{f_i}$ with
	\begin{align*}
X_{f_1}&=\{x=(x_n)\in\mathbb{T}^\mathbb{Z}: 2x_n=0\enspace(\textup{mod}\;1)\enspace\textup{for every}\enspace n\in\mathbb{Z}\},
	\\
X_{f_2}&=\{x=(x_n)\in\mathbb{T}^\mathbb{Z}: 2x_n=x_{n+1}\enspace
 (\textup{mod}\;1)\enspace\textup{for every}\enspace n\in\mathbb{Z}\},
	\\
X_{f_3}&=\{x=(x_n)\in\mathbb{T}^\mathbb{Z}: 3x_n=2x_{n+1}\enspace
 (\textup{mod}\;1)\enspace\textup{for every}\enspace n\in\mathbb{Z}\},
	\end{align*}
respectively. In each case $\alpha _{f_i}$ is
 ergodic and expansive.

Clearly, $\alpha _{f_1}$ is the full
 two-shift. For $i=2,3$ we define surjective group homomorphisms $\phi _i\colon
 X_{f_i}\to \mathbb{T}$ by $\phi _i(x)=x_0$ for every
 $x=(x_n)\in X_{f_i}$. Then $\phi _2\circ \alpha _{f_2}=M_2\circ \phi _2$, where
 $M_2x=2x$ for every $x\in\mathbb{T}$. In other words, $\alpha _{f_2}$ is
 multiplication by $2$ on $\mathbb{T}$, made invertible. Similarly we see
 that $\alpha _{f_3}$ corresponds to `multiplication by $3/2$' on $\mathbb{T}$.
	\end{exas}

	\begin{exam}
	\label{e:nonexpansive}
(1) Let $f=u^4-u^3-u^2-u+1$. Then $\alpha _f$ is conjugate to the matrix\vspace{-2mm}
	\begin{equation}
	\label{eq:companion2}
M_f= \left[
	\begin{smallmatrix}
\hphantom{-}0&1&0&0
	\\
\hphantom{-}0&0&1&0
	\\
\hphantom{-}0&0&0&1
	\\
-1&1&1&1
	\end{smallmatrix}
\right],
	\end{equation}
acting on $\mathbb{T}^4$. Put $v=u+u^{-1}$ and consider the polynomial $g(v)=v^2-v-3=u^{-2}f$ with the roots $\zeta _\pm=\frac12\pm \sqrt\frac52$. Since $\zeta _+>2$ and $|\zeta _-|<2$, the solutions of the equation $u+u^{-1}=\zeta _+$ are of the form $\theta ,\theta ^{-1}$ with $\theta >1$, and the solutions of $u+u^{-1}=\zeta _-$ are conjugate complex numbers of absolute value 1. Hence $\theta $ is a \textit{Salem number} and the automorphism $\alpha _f$ is nonexpansive and ergodic.

\smallskip (2) Let $f=5u^2-6u+5$. The roots of $f$ have the form $\frac35 \pm i\cdot \frac45$ and are both of absolute value 1. In particular, $\alpha _f$ is nonexpansive, but certainly ergodic. This example already appears in \cite{Lind 82}.

\smallskip There are, of course, irreducible noncyclotomic polynomials of arbitrarily high degree, all of whose roots have absolute value 1. Here are some more examples:\vspace{-4mm}
	\begin{gather*}
f = 2u^2 - u + 2,\\
f = 2u^4 - u^3 + 2u^2 - u + 2,\\
f = 2u^6 - 2u^5 + 4u^4 - 3u^3 + 4u^2 - 2u + 2.
	\end{gather*}
	\end{exam}

We conclude this section by recalling basic facts about entropy of (irreducible) automorphisms.

	\begin{theo}[Entropy]
	\label{t:entropy}
For every nonzero $f\in R_1$ of the form \eqref{eq:f}, the topological entropy of $\alpha _f$ coincides with the measure-theoretic entropy $h_{\lambda _{X_f}}(\alpha _f)$ of $\alpha _f$ w.r.t. the normalized Haar measure $\lambda _{X_f}$ of $X_f$, and is given by
	\begin{equation}
	\label{eq:entropy}
h(\alpha _f) = \log |f_m| + \sum\nolimits_{\theta \in \Theta _f^+}\log |\theta | = \int_0^1 \log \,|f(e^{2\pi it})|\,dt.
	\end{equation}
If $\alpha _f$ is ergodic, $\lambda _{X_f}$ is the unique measure of maximal entropy of $\alpha _f$.
	\end{theo}

	\begin{proof}
Equation \eqref{eq:entropy} is due to S.A. Yuzvinskii \cite{Yuzvinskii} (see also \cite{Lind+Ward} and \cite{LSW}). The uniqueness of the measure of maximal entropy was proved by K. Berg \cite{Berg}.
	\end{proof}

For toral automorphisms, \eqref{eq:entropy} gives the familiar expression of entropy in terms of the `large' roots of $f$. For the Examples \ref{e:x2} we obtain that $h(\alpha _{f_1})=h(\alpha _{f_2})=\log 2$ and $h(\alpha _{f_3})=\log 3$. For the polynomials $f$ in Example \ref{e:nonexpansive} (2), $h(\alpha _f)$ is the logarithm of the leading coefficient of $f$ (i.e., $\log 5$ or $\log 2$).

	\begin{theo}[Entropy and Periodic Points]
	\label{t:periodic}
For every positive integer $k$ we denote by $P_k(\alpha _f) = \{x\in X_f:\alpha _f^kx=x\}$ the set of periodic points of $\alpha _f$, with period $k$. For $0\ne f\in R_1$, $h(\alpha _f)=\lim_{k\to\infty }\frac1k\,\log\, |P_k(\alpha _f)|$.
	\end{theo}

	\begin{proof}
See \cite[p. 129]{marcus} and \cite[Sec. 4]{Lind}. For a discussion of the connection between entropy and the logarithmic growth-rate of periodic points in a more general context we refer to \cite{LSV1, LSV2}.
	\end{proof}

\section{Homoclinic points of irreducible group automorphisms}
	\label{s:homoclinic}

\emph{For the remainder of this article we assume that the irreducible polynomial $f$ in \eqref{eq:f} is noncyclotomic.}

\subsection{Linearization}

We denote by $\| \cdot \|_1$ and $\| \cdot \|_\infty $ the norms on the Banach spaces $\ell ^ 1(\mathbb{Z},\mathbb{R})$ and $\ell ^ \infty (\mathbb{Z},\mathbb{R})$ and write $\ell ^ 1(\mathbb{Z},\mathbb{Z}) \subset \ell ^ 1(\mathbb{Z},\mathbb{R})$ and $\ell ^ \infty (\mathbb{Z},\mathbb{Z}) \subset \ell ^ \infty (\mathbb{Z},\mathbb{R})$ for the subgroups of integer-valued functions. By viewing every $h\in R_1$ of the form \eqref{eq:h} as the element $(h_n)\in \ell ^ 1(\mathbb{Z},\mathbb{Z})$ we can identify $R_1$ with $\ell ^ 1(\mathbb{Z},\mathbb{Z})$.

Continuity of maps on $\ell ^ \infty (\mathbb{Z},\mathbb{R})$ will usually be understood with respect to the \textit{bounded weak$^*$-topology}, i.e., the strongest topology on $\ell ^ \infty (\mathbb{Z},\mathbb{R})$ which induces the weak$^*$-topology (or, equivalently, the topology of coordinate-wise convergence) on bounded subsets of $\ell ^ \infty (\mathbb{Z},\linebreak[0]\mathbb{Z})$. In this topology $\ell ^ \infty (\mathbb{Z},\mathbb{Z})$ is a closed subgroup of $\ell ^\infty (\mathbb{Z},\mathbb{R})$, and the shift $\bar{\sigma }\colon \ell ^\infty (\mathbb{Z},\mathbb{R})\longrightarrow \ell ^\infty (\mathbb{Z},\mathbb{R})$, defined as in \eqref{eq:shift} by
	\begin{equation}
	\label{eq:barsigma}
\bar{\sigma }(w)_n=w_{n+1},
	\end{equation}
is a homeomorphism. For every $r\ge0$, the sets
	\begin{equation}
	\label{eq:Br}
	\begin{gathered}
\bar{B}_r(\ell ^ \infty (\mathbb{Z},\mathbb{R}))=\{ w \in \ell ^ \infty (\mathbb{Z},\mathbb{R}):\| w \|_\infty \le r \},
	\\
\bar{B}_r(\ell ^ \infty (\mathbb{Z},\mathbb{Z}))=\bar{B}_r(\ell ^ \infty (\mathbb{Z},\mathbb{R}))\cap \ell ^ \infty (\mathbb{Z},\mathbb{Z})
	\end{gathered}
	\end{equation}
are compact and shift-invariant.As in \eqref{eq:hsigma} we set
	\begin{equation}
	\label{eq:hsigmabar}
\smash{h(\bar\sigma )=\sum\nolimits_{k \in \mathbb{Z}}h_k \bar\sigma ^ k\colon \ell ^ \infty (\mathbb{Z},\mathbb{R})\longrightarrow \ell ^ \infty (\mathbb{Z},\mathbb{R})}
	\end{equation}
for every $h=\sum_{k \in \mathbb{Z}}h_ku ^ k \in R_1$.

\smallskip Consider the continuous, surjective, shift-equivariant group homomorphism $\rho \linebreak[0]\colon \ell ^ \infty (\mathbb{Z},\mathbb{R})\linebreak[0]\longrightarrow \mathbb{T}^ \mathbb{Z}$ given by
	\begin{equation}
	\label{eq:rho}
\rho (w)_n=w_n\;(\textup{mod}\;1),\enspace n\in \mathbb{Z},
	\end{equation}
for every $w=(w_n)\in \ell ^ \infty (\mathbb{Z},\mathbb{R})$. The shift-invariant subgroup $X_f\subset \mathbb{T}^\mathbb{Z}$ in \eqref{eq:principal} gives rise to two shift-invariant groups of $\ell ^\infty (\mathbb{Z},\mathbb{R})$ which play an important role in the discussion of symbolic covers or representations of the automorphism $\alpha _f$:
	\begin{equation}
	\label{eq:Wf}
	\begin{gathered}
W_f\coloneqq\rho ^{-1}(X_f)=f(\bar{\sigma })^{-1}(\ell ^ \infty (\mathbb{Z},\mathbb{Z}))\subset \ell ^ \infty (\mathbb{Z},\mathbb{R}),
	\\
V_f\coloneqq f(\bar{\sigma })(W_f)\subset \ell ^ \infty (\mathbb{Z},\mathbb{Z}).
	\end{gathered}
	\end{equation}
Note that $W_f$ is closed and contains $\ker \rho =\ell ^ \infty (\mathbb{Z},\mathbb{Z})$. The kernel
	\begin{equation}
	\label{eq:Wf0}
W_f ^\circ =\ker f(\bar{\sigma })\subset W_f
	\end{equation}
of $f(\bar{\sigma })$ is obviously finite-dimensional, and $\bar{\sigma }$ is linear on $W_f^\circ$. Hence $\bar{\sigma }$ has a nonzero eigenvector in the complexification $\mathbf{W}_f^\circ=\mathbb{C} \otimes_\mathbb{R}W_f ^\circ $ of $W_f ^\circ $ with eigenvalue $\theta \in \mathbb{C}$, say, where $f(\theta )=0$. As $\bar{\sigma }$ is an isometry on $\mathbf{W}_f ^\circ $ we conclude that $\theta \in \Theta _f ^\circ $. Conversely, if $\theta \in \Theta _f ^\circ $, we set $v_n=\theta ^ n$ for every $n \in \mathbb{Z}$ and obtain that $v=(v_n)\in \mathbf{W}_f ^\circ $. This shows that $W_f ^\circ$ is the linear span of the vectors $\{ \textup{Re} (w(\theta )),\textup{Im} (w(\theta )):\theta \in \Theta _f ^\circ \}$ with
	\begin{equation}
	\label{eq:womega}
w(\theta )_n=\theta ^ n, \qquad \textup{Re} (w(\theta ))_n=\textup{Re}(\theta ^ n), \qquad \textup{Im} (w(\theta ))_n = \textup{Im} (\theta ^ n)
	\end{equation}
for every $n \in \mathbb{Z}$ and $\theta \in \Theta _f ^\circ $, where $\textup{Re}$ and $\textup{Im}$ denote the real and imaginary parts. Note that
	\begin{equation}
	\label{eq:X0}
X_f^\circ=\rho (\ker f(\bar{\sigma }))=\rho (W_f ^\circ )
	\end{equation}
is an $\alpha _f$-invariant subgroup of $X_f$. The irreducibility of $\alpha _f$ implies that the closure of $X_f ^\circ$ is either equal to $\{ 0 \}$ (if $\alpha _f$ is expansive), or to $X_f$ (if $\alpha _f$ is nonexpansive). The group $X_f ^\circ\subset X_f$ in \eqref{eq:X0} is isomorphic to $W_f ^\circ$, since $\rho $ is injective on $W_f ^\circ$, and $\alpha _f$ acts isometrically on $X_f^\circ$.

\subsection{Homoclinic points of $\alpha _f$}

In order to determine the homoclinic points of $\alpha _f$ we write
	$$
\tfrac 1{f(u)}= \tfrac 1{f_m}\sum\nolimits_{\theta \in \Theta _f}\tfrac{b_\theta }{u-\theta }
	$$
for the partial fraction decomposition of $1/f$, where $b_\theta \in \mathbb{C}$ for every $\theta \in \Theta _f$. Define elements $w ^{\pm }$ and $w ^{\circ }$ in $\ell ^ \infty (\mathbb{Z},\mathbb{R})$ by
	\begin{equation}
	\label{eq:homoclinic pm}
	\begin{aligned}
w ^{+}_n&=
	\begin{cases}
\frac 1{f_m}\cdot \sum_{\theta \in \Theta _f ^-}&\hphantom{-}b_\theta \theta ^{n-1}\hspace{5mm}\textup{if}\enspace n\ge1,
	\\
\frac 1{f_m}\cdot \sum_{\theta \in \Theta _f ^\circ \cup \Theta _f ^+}&-b_\theta \theta ^{n-1}\hspace{5mm}\textup{if}\enspace n\le0,
	\end{cases}
	\\
w ^{-}_n&=
	\begin{cases}
\frac 1{f_m}\cdot \sum_{\theta \in \Theta _f ^-\cup \Theta _f ^\circ }&\hphantom{-}b_\theta \theta  ^{n-1}\hspace{5mm}\textup{if}\enspace n\ge1,
	\\
\frac 1{f_m}\cdot \sum_{\theta \in \Theta _f ^+}&-b_\theta \theta ^{n-1}\hspace{5mm}\textup{if}\enspace n\le0,
	\end{cases}
	\\
w ^{\circ }_n&=\hspace{2.7mm}\tfrac 1{f_m}\cdot \textstyle \sum_{\theta \in \Theta _f ^\circ }\hspace{12mm}b_\theta \theta ^{n-1}\hspace{5mm}\textup{for every}\enspace n \in \mathbb{Z}.
	\end{aligned}
	\end{equation}
Then
	\begin{equation}
	\label{eq:homoclinic2}
	\begin{gathered}
w ^{\circ }\in W_f ^\circ ,\qquad w ^{+}+w ^{\circ }=w ^{-},
	\\
f(\bar{\sigma })(w ^{+})_n=f(\bar{\sigma })(w ^{-})_n=v ^ \Delta _n\coloneqq
	\begin{cases}
1&\textup{if}\enspace n=0,
	\\
0&\textup{otherwise}.
	\end{cases}
	\end{gathered}
	\end{equation}
The points $w ^{\pm }\in \ell ^ \infty (\mathbb{Z},\mathbb{R})$ have the following properties.
	\begin{enumerate}
	\item[(a)]\label{a}
$w^\pm\in W_f$ by \eqref{eq:Wf} and \eqref{eq:homoclinic2}, and $x ^{\pm }\coloneqq\rho (w ^{\pm })\in X_f$ by \eqref{eq:Wf}.
	\item[(b)]\label{b}
$\lim_{n \to \infty }w ^{+}_n=\lim_{n \to \infty }w ^{-}_{-n}=0$ exponentially fast.
	\item[(c)]\label{c}
If $\alpha _f$ is nonexpansive, then $\lim_{n \to \infty }x^{+}_n=\lim_{n \to \infty }x^{-}_{-n}=0$, but neither of these points is homoclinic.
	\item[(d)]\label{d}
If $\alpha _f$ is expansive, then $w^{+}=w^{-}$ and $x ^{+}=x ^{-}$. Put $w^\Delta =w^{\pm }$ and $x^\Delta =\rho (w^\Delta )$. Then $\lim_{|n|\to\infty }w^\Delta _n=0$ and $x^\Delta $ is homoclinic.
	\end{enumerate}

\subsubsection{The expansive case}

	\smallskip \begin{theo}
	\label{t:expansive}
If $\alpha _f$ is expansive, then every homoclinic point of $\alpha _f$ is of the form $x=h(\alpha _f)(x^\Delta )$ for some $h\in R_1$, where $x^\Delta =\rho (w^{+}) =\rho (w^{-})$. In other words, the homoclinic group $\Delta _{\alpha _f}(X_f)$ of $\alpha _f$ is the subgroup of $X_f$ generated by the orbit of the homoclinic point $x^\Delta $.
	\end{theo}

	\begin{proof}
If $\alpha _f$ is expansive, then $f(\bar{\sigma })(w^\Delta )=v^\Delta $ by \eqref{eq:homoclinic2}. If a point $x\in X_f$ is homoclinic, we can lift it to a point $w\in W_f$ satisfying that $\rho (w)=x$ and $\lim_{|n|\to\infty }w_n = 0$. By \eqref{eq:Wf}, $h\coloneqq f(\bar{\sigma })(w)\in\ell ^\infty (\mathbb{Z},\mathbb{Z})$, which implies that $h\in \ell ^1(\mathbb{Z},\mathbb{Z})=R_1$.

We write $h$ as $h=\sum_{n\in \mathbb{Z}}h_nu^n$ and set $h^*=h(u^{-1})=\sum_{n\in \mathbb{Z}}h_{-n}u^n$. The points $w$ and $w'=h^*(\bar{\sigma })(w^\Delta )$ satisfy that $f(\bar{\sigma })(w)=f(\bar{\sigma })(w') = h$ and hence, since the kernel of $f(\bar{\sigma })$ is trivial, that $w=w'$ and $x=\rho (w)= \rho (h^*(\bar{\sigma })(x^\Delta ))=h^*(\alpha _f)(x^\Delta )$.
	\end{proof}

 Motivated by Theorem \ref{t:expansive} we introduce the following definition.

	\begin{defi}
	\label{d:fundamental}
A homoclinic point $w\in X_f$ of $\alpha _f$ is \textit{fundamental} if its orbit generates the group $\Delta _{\alpha _f}(X_f)$ of all homoclinic points of $\alpha _f$.
	\end{defi}

	\begin{exam}
	\label{e:alphaA}
Consider the matrix $A\in \textup{GL}(2,\mathbb{Z})$ in \eqref{eq:A} and the associated expansive automorphism $\alpha _A$ of $\mathbb{T}^2$. Which of the homoclinic points $x^\Delta ,y^\Delta ,z^\Delta ,w^\Delta $ in Figure 1 \vpageref{Figure1} is fundamental?

If $\pi \colon \mathbb{R}^2\longrightarrow \mathbb{Z}^2$ is the quotient map, then every homoclinic point $w$ of $\alpha _A$ is of the form $w=\pi (v_+ \cap (v_-+\mathbf{m}))$ for some $\mathbf{m}\in \mathbb{Z}^2$ and will therefore be denoted by $w(\mathbf{m})$. Then $\alpha _f^k(w(\mathbf{m}))=w(A^k\mathbf{m})$ for every $k\in \mathbb{Z}$. Since $\{A^k\mathbf{m}:k\in \mathbb{Z}\}$ generates $\mathbb{Z}^2$ if and only if $\mathbf{m}=(m_1,m_2)$ is primitive (i.e., if $\gcd(m_1,m_2)=1$), the points $x^\Delta =w(1,0),\,y^\Delta =w(1,1),\,z^\Delta =w(0,1)$ are fundamental, but $w^\Delta =w(0,2)$ is not.
	\end{exam}

	\begin{exas}
	\label{e:x3}
Let $f_i,\,i=1,2,3$, be the polynomials appearing in Example \ref{e:x2}. For $f_1=2$, put
	\begin{displaymath}
\smash{x^\Delta _n=
	\begin{cases}
\frac12\;(\textup{mod}\,1)&\textup{if}\;n=0,
	\\
0&\textup{otherwise}.
	\end{cases}}
	\end{displaymath}
For $f_2=2-u$, set
	\begin{displaymath}
\smash{x^\Delta _n=
	\begin{cases}
2^{n-1}\;(\textup{mod}\,1)&\textup{if}\;n\le0,
	\\
0&\textup{otherwise},
	\end{cases}}
	\end{displaymath}
and for $f_3=3-2u$, let
	\begin{displaymath}
\smash[t]{x^\Delta _n=
	\begin{cases}
\frac{3^{n-1}}{2^n}\;(\textup{mod}\,1)&\textup{if}\;n\le0,
	\\
0&\textup{otherwise}.
	\end{cases}}
	\end{displaymath}
In all these cases $x^\Delta $ is fundamental homoclinic for $\alpha _{f_i}$.

\smallskip For $f=2-3u$, the point
	\begin{displaymath}
\smash{x^\Delta _n=
	\begin{cases}
\frac{2^{n}}{3^{n+1}}\;(\textup{mod}\,1)&\textup{if}\;n\ge0,
	\\
0&\textup{otherwise},
	\end{cases}}\vspace{-1mm}
	\end{displaymath}
is fundamental homoclinic.
	\end{exas}

\subsubsection{The nonexpansive case}

	\smallskip\begin{theo}
	\label{t:nonexpansive}
If $\alpha _f$ is nonexpansive, it has no nonzero homoclinic points.
	\end{theo}

	\begin{proof}
Suppose that $\alpha _f$ is nonexpansive, but that there exists a nonzero point $x\in X_f$ satisfying \eqref{eq:homoclinic}. We can lift $x$ to a point a point $w\in W_f$ satisfying $\rho (w)=x$ and $\lim_{|n|\to\infty }w_n = 0$ (as in the expansive case). Then $h\coloneqq f(\bar{\sigma })(w)\in\ell ^\infty (\mathbb{Z},\mathbb{Z})\in R_1$, and the point $\tilde{w}^-\coloneqq h^*(\bar{\sigma })(w^{-})$ satisfies that $f(\bar{\sigma })(\tilde{w}^-)= f(\bar{\sigma })(w)$, so that $w-\tilde{w}^-\in \ker f(\bar{\sigma })=W_f^\circ$.

From property (b) of $w^-$ \vpageref{b} we know that $\lim_{n \to-\infty }\tilde{w}^-_n=\lim_{n \to-\infty }\linebreak[0]w_n=0$. Since $w-\tilde{w}\in W_f ^\circ$ and every point in $W_f^\circ$ is almost periodic, we conclude that $w=\tilde{w}^-$. However,
	$$
\tilde{w}_n^-=\tfrac 1{f_m}\sum\nolimits_{\theta \in \Theta _f ^\circ \cup \Theta _f ^-}b_\theta \theta ^{n-1}h(\theta )
$$
for all sufficiently large positive $n$, which shows that
\begin{equation}
\label{eq:sum}
\sum\nolimits_{\theta \in \Theta _f ^\circ }b_\theta \theta ^ nh(\theta )=0\enspace \textup{for every}\enspace n\ge0.
\end{equation}
From \eqref{eq:sum} we see that $h(\theta )=0$ for every $\theta \in \Theta _f ^\circ $ or, equivalently, that $h$ is divisible by $f$. We set $h=fh'$ with $h'\in R_1$, $v'=h'(\bar{\sigma })(v ^ \Delta )\in \ell ^ \infty (\mathbb{Z},\mathbb{Z})$ and $\tilde{w}'=h'(\bar{\sigma })(w ^{-})\in W_f$ as above, and conclude that $w=\tilde{w}^-=f(\bar{\sigma })(\tilde{w}')\in \ell ^ \infty (\mathbb{Z},\mathbb{Z})$ and $x=\rho (w)=0$, contrary to our choice of $x$.
	\end{proof}

Although $\alpha _f$ has no nonzero homoclinic points if it is not expansive, we have at our disposal the `one-sided homoclinic' points $x^\pm$ in property (a) \vpageref{a}. As in Example \ref{e:alphaA} it may be helpful to identify these points in a few special cases.

	\begin{exam}
	\label{e:salem}
Consider the automorphism $\alpha =\alpha _{M_f}$ of $\mathbb{T}^4$ defined by the nonhyperbolic matrix $M_f$ in \eqref{eq:companion2}. We use the notation of Example \ref{e:nonexpansive} and write $W^+$ and $W^-$ for the the one-dimensional eigenspaces of $M_f$ corresponding to the eigenvalues $\theta $ and $\theta ^{-1}$, and $W^\circ$ for the two-dimensional eigenspace corresponding to the complex eigenvalues $\zeta ,\zeta ^{-1}$, on which $M_f$ acts as an irrational rotation. Although the intersections $W^+\cap (W^-+\mathbf{m})$ are empty whenever $\mathbf{m}\in \mathbb{Z}^4$ is nonzero, there exist, for every $\mathbf{m}\in \mathbb{Z}^4$, unique points $w(\mathbf{m})^+\in W ^-\cap (W^+ +W^\circ +\mathbf{m})$ and $w(\mathbf{m})^-\in (W ^- +W^\circ ) \cap (W^+ +\mathbf{m})$. If $\pi \colon \mathbb{R}^4\longrightarrow \mathbb{T}^4$ is the quotient map, the points $x(\mathbf{m})^\pm=\pi (w(\mathbf{m}^\pm))$ satisfy that $\lim_{n\to\infty }\alpha ^n(x(\mathbf{m})^+)= \lim_{n\to-\infty }\alpha ^n(x(\mathbf{m})^-)=0$.
	\end{exam}

	\begin{exam}
	\label{e:funny}
Let $f=5u^2-6u+5$ (cf. Example \ref{e:nonexpansive} (2)). An elementary recursive calculation shows that $x^\pm = \rho (w^\pm)$ with
	\begin{gather*}
w^- = (\dots, 0, \underline{0}, 0, \tfrac{1}{5}, \tfrac{6}{25}, \tfrac{11}{125}, -\tfrac{84}{625}, \dots),
	\\
w^+ = (\dots , -\tfrac{84}{625}, \tfrac{11}{125}, \tfrac{6}{25}, \underline{\tfrac{1}{5}}, 0, 0, 0, \dots ),
	\end{gather*}
where the zero-th coordinates are underlined.

\smallskip For $f=2u^2-u+2$ we obtain $x^\pm = \rho (w^\pm)$ with
	\begin{gather*}
w^- = (\dots, 0, \underline{0}, 0, \tfrac{1}{2}, \tfrac{1}{4}, -\tfrac{3}{8}, -\tfrac{7}{16}, \dots),
	\\
w^+ = (\dots, -\tfrac{7}{16}, -\tfrac{3}{8}, \tfrac{1}{4}, \underline{\tfrac{1}{2}}, 0,0,0,\dots ).
	\end{gather*}
	\end{exam}

\section{Symbolic covers of expansive automorphisms}\label{ss:covers}

Suppose that $\alpha _f$ is expansive, and that $x^\Delta =\rho (w^\Delta )$ is the fundamental homoclinic point of $\alpha _f$ described in Theorem \ref{t:expansive} (cf. Definition \ref{d:fundamental}). As in \eqref{eq:xi_w} we define maps $\bar{\xi }\colon \ell ^\infty (\mathbb{Z},\mathbb{Z})\longrightarrow W_f$ and $\xi \colon \ell ^\infty (\mathbb{Z},\mathbb{Z})\longrightarrow X_f$ by setting, for every $v=(v_n)\in \ell ^\infty (\mathbb{Z},\mathbb{Z})$,
	\begin{equation}
	\label{eq:xi}
	\begin{gathered}
\bar{\xi }(v)=\sum\nolimits_{n\in \mathbb{Z}}v_n\bar{\sigma }^{-n} w^\Delta \\ \xi (v)=\rho \circ \bar{\xi }(v)=\sum\nolimits_{n\in \mathbb{Z}}v_n\alpha _f^{-n} x^\Delta .
	\end{gathered}
	\end{equation}

	\begin{theo}
	\label{t:expansive2}
\textup{(1)} The map $\bar{\xi }\colon \ell ^\infty (\mathbb{Z},\mathbb{Z})\longrightarrow W_f$ is a continuous and shift-equi\-variant group homomorphism.

\textup{(2)} For every $v\in \ell ^\infty (\mathbb{Z},\mathbb{Z})$, $f(\bar{\sigma })\circ \bar{\xi }(v)=v$. Hence $V_f=\ell ^\infty (\mathbb{Z},\mathbb{Z})$, and the maps $\bar{\xi }\colon V_f\longrightarrow W_f$ and $f(\bar{\sigma })\colon W_f\longrightarrow V_f$ are continuous group isomorphisms which are inverse to each other.

\textup{(3)} The map $\xi \colon \ell ^\infty (\mathbb{Z},\mathbb{Z})\longrightarrow X_f$ is a continuous, surjective and shift-equivariant group homomorphism with kernel $f(\bar{\sigma })(\ell ^\infty (\mathbb{Z},\mathbb{Z}))$.
	\end{theo}

	\begin{proof}
The exponential decay of the coordinates of $w^\Delta $ guarantees that $\bar{\xi }$ is a well-defined and continuous group homomorphism which is obviously shift-equivariant. The second equation in \eqref{eq:homoclinic2} shows that $f(\bar{\sigma })\circ \bar{\xi }=\textup{Id}_{\ell ^\infty (\mathbb{Z},\mathbb{Z})}$, which proves (2). For the proof of (3) it suffices to note that $\ker \rho =\ell ^\infty (\mathbb{Z},\mathbb{Z})$.
	\end{proof}

One can use Theorem \ref{t:expansive2} for the construction of symbolic covers of $(X_f,\alpha _f)$ \`a la Vershik. We start with a basic observation.

	\begin{prop}
	\label{p:cover1}
Let $V=\bar{B}_{\|f\|_1}(\mathbb{Z},\mathbb{Z})=\{v\in\ell ^\infty (\mathbb{Z},\mathbb{Z}):\|v\|_\infty \le \|f\|_1\}$ \textup{(}cf. \eqref{eq:Br}\textup{)}. Then $\xi (V)=X_f$.
	\end{prop}

	\begin{proof}
Since $\rho (W_f\cap [0,1]^\mathbb{Z})=X_f$ and $f(\bar{\sigma })(W_f\cap [0,1]^\mathbb{Z})\subset V$, we can find, for every $x\in X_f$, a $w\in W_f\cap [0,1]^\mathbb{Z}$ with $\rho (w)=x$ and $v\coloneqq f(\bar{\sigma })(w)\in V$. Then $\xi (v)=x$.
	\end{proof}

We have thus found a symbolic cover $(V,\bar{\sigma })_\xi $ of $(X_f,\alpha _f)$. However, since the entropies of the restriction of $\bar{\sigma }$ to $V$ and of $\alpha _f$ satisfy that $h(\bar{\sigma }|_V)=\log (2\|f\|_1-1)>h(\alpha _f) = \int_0^1\log|f(e^{2\pi is})|\,ds $ by Theorem \ref{t:entropy}, $(V,\bar{\sigma })_\xi $ is not an equal entropy cover of $(X_f,\alpha _f)$. Nevertheless even the simple-minded statement of Proposition \ref{p:cover1} yields a very strong specification property of $\alpha _f$.

	\begin{coro}
	\label{c:specification}
If $d$ is a metric on $X_f$ then there exists, for every $\varepsilon >0$, an integer $N(\varepsilon )>0$ with the following properties.

\textup{(1)} Let $\mathcal{I}$ be a finite or infinite collection of disjoint nonempty subsets of $\mathbb{Z}$ such that the distances $|I-I'|=\min \{|m-n|:m\in I,\,n\in I'\}$ are $\ge N(\varepsilon )$ whenever $I,I'$ are distinct elements of $\mathcal{I}$. Then there exists, for every collection $(x_I)_{I\in \mathcal{I}}$ of elements of $X_f$, a point $y\in X_f$ with $d(\alpha _f^kx_I,\alpha _f^ky)<\varepsilon $ for every $I\in \mathcal{I}$ and $k\in I$.

\textup{(2)} If $J\coloneqq \bigcup_{I\in \mathcal{I}}I$ is finite, the the point $y$ in \textup{(1)} can be chosen to be periodic with any period $p\ge N(\varepsilon )+\max\{|k-l|:k,l\in J\}$.
	\end{coro}

In order to construct \textit{equal entropy} covers of $(X_f,\alpha _f)$ we consider the lexicographic order $\prec$ on $\ell ^1(\mathbb{Z},\mathbb{Z})$ (defined by putting $h \succ 0$ if the leading term in $h$ is positive) and set, for every \textit{SFT} $W\subset V$ with $\xi (W)=X_f$,
	\begin{equation}
	\label{eq:W*}
W^*= W \smallsetminus \bigcup\nolimits_{\{h\in \ell ^1(\mathbb{Z},\mathbb{Z}):h \succ 0\}}(W+f(\bar{\sigma })h).
	\end{equation}

	\begin{prop}[\protect{\cite[Proposition 4.2]{S3}}]
	\label{p:W*}
Let $W\subset V$ be a transitive \textit{SFT} with $\xi (W)=X_f$, and let $W^*\subset W$ be defined by \eqref{eq:W*}. Then $W^*$ is a mixing sofic shift with $\xi (W^*)=X_f$ and $h(\bar\sigma |_{W^*})=h(\alpha _f)$. Furthermore, the restriction of $\xi $ to $W^*$ is bounded-to-one.

If there exists a fixed point $\mathbf{c}\in W$ of $\bar{\sigma }$ with
	\begin{equation}
	\label{eq:unique}
\xi ^{-1}(\{\xi (\mathbf{c})\})\cap W=\{\mathbf{c}\},
	\end{equation}
then the restriction of $\xi $ to $W^*$ is almost one-to-one.
	\end{prop}

As in \cite[Proposition 5.2]{S3} one can find a \textit{SFT} $W\subset V$ with the properties required by Proposition \ref{p:W*}. Indeed, since $\alpha _f$ is expansive, it has only finitely many fixed points, $\kappa \ge 1$, say. Then
	\begin{displaymath}
D=\{t\in \mathbb{R}: (\dots ,t,t,t,\dots )\in W_f\} = \tfrac1\kappa \mathbb{Z}.
	\end{displaymath}
We put $J=[-1/2\kappa ,1-1/2\kappa ]\subset \tilde{J}=(-5/8\kappa ,1-3/8\kappa )\subset \mathbb{R}$ and observe that $\rho (W_f\cap J^\mathbb{Z})=X_f$, and that $W_f\cap \tilde{J}^\mathbb{Z}$ contains exactly $\kappa $ elements of $D$. As explained in the proof of \cite[Proposition 5.2]{S3}, there exists a \textit{SFT} $W\subset V$ such that $W_f\cap J^\mathbb{Z}\subset \bar{\xi }(V) \subset W_f\cap \tilde{J}^\mathbb{Z}$. Then $\xi (W)=X_f$, and every fixed point of $\alpha _f$ has a unique pre-image under $\xi $ in $W$. This proves our next corollary.

	\begin{coro}[\protect{\cite[Theorem 5.1]{S3}}]
	\label{c:1-1-cover}
Suppose that $\alpha _f$ is expansive. Then there exists a mixing sofic shift $W^*\subset \ell ^\infty (\mathbb{Z},\mathbb{Z})$ such that $\xi (W^*)=X_f$ and the restriction of $\xi $ to $W^*$ is injective on the set of doubly transitive points in $W^*$.
	\end{coro}

Although one can determine sofic covers of the form \eqref{eq:W*} explicitly --- at least in simple examples --- there are no `natural' symbolic covers or representations of $(X_f,\alpha _f)$, unless $f$ is a Pisot polynomial.\footnote{\label{Pisot}\, An irreducible polynomial $f$ of the form \eqref{eq:f} is \textit{Pisot} if it has a root $\beta >1$ whose conjugates all have absolute value $<1$.} In this case the two-sided $\beta $-shift $V_\beta \subset \ell ^\infty (\mathbb{Z},\mathbb{Z})$ is an equal entropy (and hence bounded-to-one) cover of $(X_f,\alpha _f)$ (\cite{Ver3, ES, SV, S3}). As an instance of a wider class of conjectures, collectively referred to as \textit{Pisot conjecture(s)}, $(V_\beta ,\bar{\sigma })_\xi $ had been conjectured for some time to be an almost one-to-one cover of $(X_f,\alpha _f)$. After earlier partial results (cf. e.g., \cite{S3,Sidorov}), this conjecture was recently settled affirmatively by Barge in \cite{Barge}.\footnote{\,I am grateful to N. Sidorov for alerting me to this reference.}

\section{Pseudocovers of nonexpansive automorphisms}\label{s:nonexpansive}

Is there any hope of constructing symbolic covers or symbolic representations of $(X_f,\alpha _f)$ in the nonexpansive case? We start with the bad news.

	\begin{defi}
	\label{d:homoclinic relation}
Let $T$ be a homeomorphism of a compact metrizable space $Y$ with a metric $\delta $. Two points $x,y \in Y$ are \emph{homoclinic} if $\lim_{|n|\to \infty }\delta (T ^ nx,\linebreak[0]T ^ ny)=0$. The \emph{homoclinic equivalence relation} $\boldsymbol{\Delta }_T(Y)$ is defined as
	$$
\boldsymbol{\Delta }_T(Y)=\{(x,y)\in Y ^ 2:x\enspace \textup{and}\enspace y\enspace \textup{are homoclinic}\}.
	$$
For every $x \in Y$ we denote by
	$$
\boldsymbol{\Delta }_T(x)=\{ y \in Y:(x,y)\in \boldsymbol{\Delta }_T(Y)\}
	$$
the \emph{homoclinic equivalence class} of $x$. The homoclinic relation $\boldsymbol{\Delta }_T(Y)$ is \emph{topologically transitive} if $\boldsymbol{\Delta }_T(y)$ is dense in $Y$ for \emph{some} $y \in Y$, and \emph{minimal} if $\boldsymbol{\Delta }_T(y)$ is dense in $Y$ for \emph{every} $y \in Y$.

All these definitions are independent of the specific choice of the metric $\delta $.
	\end{defi}

	\begin{prop}
\label{p:bad}
Suppose that $(X_f,\alpha _f)$ is nonexpansive and that $T$ is a homeomorphism of a compact metrizable space $Y$ whose homoclinic relation $\boldsymbol{\Delta }_Y(T)$ is topologically transitive. If $\phi \colon Y \longrightarrow X$ is a continuous $(T,\alpha _f)$-equivariant map then $\phi (Y)$ consists of a single fixed point $\bar{x}$ of $\alpha $ in $X_f$.
	\end{prop}

	\begin{proof}
For any homoclinic pair $(y,y')$ in $Y$, the points $\phi (y)$ and $\phi (y')$ are homoclinic in $X_f$, and hence $\phi (y)-\phi (y')\in \Delta _{\alpha _f}(X_f)=\{0\}$ by Theorem \ref{t:nonexpansive}. If $\boldsymbol{\Delta }_T(y)$ is dense in $Y$ for some $y \in Y$ then the continuity of $\phi $ implies that $\phi (Y)$ is a single point which must be fixed under $\alpha _f$.
	\end{proof}

	\begin{coro}
	\label{c:bad}
Let $Y$ be a mixing \textit{SFT} with finite or countably infinite alphabet. If $(X_f,\alpha _f)$ is nonexpansive, then every continuous equivariant map $\phi \colon Y \longrightarrow X_f$ sends $Y$ to a single point.
	\end{coro}

	\begin{proof}
If $T$ is the shift on $Y$, then $\boldsymbol{\Delta }_Y(T)$ is minimal, and our claim follows from Proposition \ref{p:bad}.
	\end{proof}

	\begin{coro}
	\label{c:bad2}
Let $Y$ be a topologically transitive sofic shift with finite alphabet. Then every continuous equivariant map $\phi \colon Y \longrightarrow X_f$ sends $Y$ to a finite set.
	\end{coro}

	\begin{proof}
Since $Y$ is a continuous equivariant image of a topologically transitive \textit{SFT} with finite alphabet, the result follows from Corollary \ref{c:bad}.
	\end{proof}

	\begin{rema}
	\label{r:bad}
Some non-sofic shift-spaces have a topologically transitive homoclinic equivalence relation. For example, if $\beta >1$ is a real number, and if $V_\beta \subset \{ 0,\dots ,\lceil \beta -1 \rceil \} ^ \mathbb{Z}$ is the two-sided beta-shift space, then the homoclinic equivalence relation $\boldsymbol{\Delta }_{\bar \sigma}(V_\beta )$ of the beta-shift $\sigma _\beta $ is topologically transitive, although $V_\beta $ is in general not sofic.
	\end{rema}

The Corollaries \ref{c:bad} -- \ref{c:bad2} and Remark \ref{r:bad} imply that $\alpha _f$ cannot have Markov \textup{(}or sofic\textup{)} partitions or covers in any reasonable sense, if it is nonexpansive. However, by a general result in \cite[Theorem 7.4]{BFF} or \cite[Theorem 8.6]{BD}, $(X_f,\alpha _f)$ has an equal entropy symbolic cover; regrettably, nothing much can be said about how `nice' such a cover could be.\footnote{\label{entexp}\, A sufficient condition for a homeomorphism $T$ of a compact metric space $X$ to have an equal entropy symbolic cover is that $T$ be \textit{asymptotically $h$-expansive} (cf. \cite[p. 720]{BFF}), a condition which can be verified quite easily for $\alpha _f$.}

\smallskip In an attempt to imitate --- at least in spirit --- the construction of symbolic covers in the expansive case we set, for every $v=(v_n)\in \ell ^ \infty (\mathbb{Z},\mathbb{Z})$ and $k\in \mathbb{Z}$,
	\begin{equation}
	\label{eq:barxi*}
	\begin{aligned}
\bar{\xi }^*(v)_k&=\sum\nolimits_{n\ge0}v_n w^-_{k-n}+\sum\nolimits_{n<0}v_n w^+_{k-n}
	\\
&= \Bigl[\sum\nolimits _{n\ge0}v_n \bar{\sigma }^{-n}(w ^-)+\sum\nolimits_{n<0}v_n \bar{\sigma }^{-n}(w ^+)\Bigr]_k.
	\end{aligned}
	\end{equation}
Since the coordinates $w_n ^+$ and $w_{-n}^-$ decay exponentially as $n \to \infty $ by property (b) \vpageref{b}, $\bar{\xi }^*(v)_k$ converges for every $k\in \mathbb{Z}$, and for every $k\in \mathbb{Z}$ the map $v\mapsto \bar{\xi }^*(v)_k$ is continuous on $\ell ^\infty (\mathbb{Z},\mathbb{Z})$. If we put
	\begin{equation}
	\label{eq:l*}
\ell ^*(\mathbb{Z},\mathbb{R})=\bigl\{ w=(w_n) \in \mathbb{R}^ \mathbb{Z}: \sup\nolimits_{n \in \mathbb{Z}} \tfrac{|w_n|}{|n|+1}<\infty \bigr\} \supset \ell ^ \infty (\mathbb{Z},\mathbb{R}),
	\end{equation}
then \eqref{eq:barxi*} shows that $\bar{\xi }^*(v) = (\bar{\xi }^*(v)_k)_{k\in \mathbb{Z}} \in \ell ^*(\mathbb{Z},\mathbb{R})$ for every $v\in \ell ^\infty (\mathbb{Z},\mathbb{Z})$.

We extend the maps $\bar{\sigma },f(\bar{\sigma })$ and $\rho $ to $\ell ^*(\mathbb{Z},\mathbb{R})$ in the obvious manner and note that the kernel of $f(\bar{\sigma })$ does not change with this extension:
	\begin{equation}
	\label{eq:kernel}
\{w\in \ell ^*(\mathbb{Z},\mathbb{R}): f(\bar{\sigma })(w)=0\}=W_f^\circ.
	\end{equation}

	\begin{theo}[\cite{LiS2}]
	\label{t:boundedness}
Let $W_f\subset \ell ^\infty (\mathbb{Z},\mathbb{R})$ and $V_f=f(\bar{\sigma })(W_f)\subset \ell ^\infty (\mathbb{Z},\mathbb{Z})$ be given by \eqref{eq:Wf}, and define $\bar{\xi }^*\colon \ell ^\infty (\mathbb{Z},\mathbb{Z})\longrightarrow \ell ^*(\mathbb{Z},\mathbb{R})$ by \eqref{eq:barxi*} -- \eqref{eq:l*}.

\smallskip\textup{(1)} $f(\bar{\sigma })\circ \bar{\xi }^*(v)=v$\enspace and\enspace $\bar{\xi }^*\circ f(\bar{\sigma })(w)-w \in W_f ^\circ$ for every $v \in \ell ^\infty (\mathbb{Z},\mathbb{Z})$ and $w \in W_f $.

\smallskip\textup{(2)} $V_f=\{v\in \ell ^\infty (\mathbb{Z},\mathbb{Z}):\bar{\xi }^*(v)\in \ell ^\infty (\mathbb{Z},\mathbb{R})\}$\enspace and\enspace $\bar{\xi }^*(V_f)\subset W_f$.

\smallskip\textup{(3)} There exists a constant $c>0$ with $\|\bar{\xi }^*\circ f(\bar{\sigma })(w)\|_\infty \le c \cdot \|w\|_\infty$ for every $w\in W_f$.
	\end{theo}

	\begin{proof}
The proof of these statements is taken from \cite{LiS2}. Equation \eqref{eq:homoclinic2} shows that $f(\bar{\sigma })(\bar{\xi }^*(v))=v$ for every $v\in \ell ^1(\mathbb{Z},\mathbb{Z})$. If $v\in \ell ^\infty (\mathbb{Z},\mathbb{Z})$ we define $v^{(n)}\in \ell ^1(\mathbb{Z},\mathbb{Z}),\;n\ge1$, by setting
	\begin{displaymath}
v^{(n)}_k=
	\begin{cases}
v_k&\textup{if}\;|k|\le n,
	\\
0&\textup{otherwise}.
	\end{cases}
	\end{displaymath}
Then $f(\bar{\sigma })(\bar{\xi }^*(v^{(n)}))=v^{(n)}$ for every $n\ge1$, and by letting $n\to\infty $ and observing that $\bar{\xi }^*(v^{(n)})\to \bar{\xi }^*(v)$ coordinate-wise as $n\to\infty $, we obtain the first identity in (1) for every $v\in \ell ^\infty (\mathbb{Z},\mathbb{Z})$. The second relation in (1) follows from the first by setting $v=f(\bar{\sigma })(w)$ for $w\in W_f$: since $f(\bar{\sigma })(\bar{\xi }^*(f(\bar{\sigma })(w)))=f(\bar{\sigma })(w)$, we obtain that $\bar{\xi }^*(v) = \bar{\xi }^*\circ f(\bar{\sigma })(w)-w\in W_f^\circ$ by \eqref{eq:kernel}. In particular, $\bar{\xi }^*(v)$ differs from $w$ by an element of $W_f^\circ$ and thus lies in $\ell ^\infty (\mathbb{Z},\mathbb{R})$, which proves that $\bar{\xi }^*(V_f)\subset W_f$.

We have shown that $V_f\subset \{v\in \ell ^\infty (\mathbb{Z},\mathbb{Z}):\bar{\xi }^*(v)\in \ell ^\infty (\mathbb{Z},\mathbb{R})\}$. For the reverse inclusion we take $v\in \ell ^\infty (\mathbb{Z},\mathbb{Z})$ such that $\bar{\xi }^*(v)\in \ell ^\infty (\mathbb{Z},\mathbb{R})$ and note that $f(\bar{\sigma })(\bar{\xi }^*(v))=v\in \ell ^\infty (\mathbb{Z},\mathbb{Z})$ by (1). Then $\bar{\xi }^*(v)\in W_f$ by \eqref{eq:Wf}, and hence $v=f(\bar{\sigma })(\bar{\xi }^*(v))\in V_f$. This completes the proof of (2).

Finally we turn to (3). From \eqref{eq:barxi*} it is clear there exists a constant $c'>0$ with $|\bar{\xi }^*(v)|_n\le c'\cdot \| v \|_\infty $ for every $v \in \ell ^ \infty (\mathbb{Z},\mathbb{Z})$ and $n=0,\dots ,m-1$, where $m = \textup{deg}(f)$ is the degree of $f$. Hence
	$$
|\bar{\xi }^*\circ f(\bar{\sigma })(w)_n - w_n|\le c' \| f(\bar{\sigma })(w)\|_\infty + \|w\|_\infty \le (c' \| f \|_1 +1)\cdot \| w \|_\infty
	$$
for every $w \in \ell ^ \infty (\mathbb{Z},\mathbb{R})$ and $n=0,\dots ,m-1$. Since there exists a constant $c''>0$ with
\begin{equation}
\label{eq:bound2}
\| w \|_\infty \le c''\cdot \max\,\{|w_0|,\dots ,|w_{m-1}|\}
\end{equation}
for every $w \in W_f ^\circ $ by \eqref{eq:womega}, the second relation in (1) allows us to find a constant $c>0$ with
	\begin{equation}
	\label{eq:bound}
\| \bar{\xi }^*\circ f(\bar{\sigma })(w)\|_\infty \le \| \bar{\xi }^*\circ f(\bar{\sigma })(w)-w\|_\infty + \|w\|_\infty\le c \cdot \| w \|_\infty
	\end{equation}
for every $w \in W_f$.
	\end{proof}

As we shall see later, the space $V_f$ plays an important role in the search for anything resembling symbolic covers or symbolic representations of $(X_f,\alpha _f)$ in the nonexpansive case. The following corollary of Theorem \ref{t:boundedness} (2) gives a little bit of information about this somewhat elusive object.

	\begin{coro}
	\label{c:boundedness}
\textup{(1)} An element $v\in \ell ^\infty (\mathbb{Z},\mathbb{Z})$ lies in $V_f$ if and only if
	\begin{equation}
	\label{eq:boundedness}
\smash[b]{\sup\nolimits_{m,n\ge0}\, \Bigl|\sum\nolimits_{k=-m}^{n} v_k\theta ^k\Bigr| < \infty}
	\end{equation}
for every $\theta \in \Theta _f^\circ $.

\smallskip\textup{(2)} $V_f\supset \ell ^1(\mathbb{Z},\mathbb{Z})$.

\smallskip\textup{(3)} $V_f$ contains every periodic element of $\ell ^\infty (\mathbb{Z},\mathbb{Z})$. In particular, if $V\subset \ell ^\infty (\mathbb{Z},\mathbb{Z})$ is a \textit{SFT}, then $V\cap V_f$ is dense in $V$.

\smallskip\textup{(4)} If $r$ is a positive integer, then there exists, for every $\varepsilon >0$, a subshift $V\subset \{0,\dots ,r\}^\mathbb{Z}\cap V_f$ with topological entropy $\ge r+1-\varepsilon $.
	\end{coro}

	\begin{proof}
For the proof of (1) we define an element $y\in \ell ^1(\mathbb{Z},\mathbb{R})$ by setting
	\begin{displaymath}
y_n=
	\begin{cases}
w_n^+&\textup{if}\enspace n\ge1,
	\\
w_n^-&\textup{if}\enspace n\le0.
	\end{cases}
	\end{displaymath}
Clearly, $\sum_{n\in \mathbb{Z}}v_n \bar{\sigma }^{-n}y \in \ell ^\infty (\mathbb{Z},\mathbb{R})$, and $\bar{\xi }^*(v)\in\ell ^\infty (\mathbb{Z},\mathbb{R})$ if and only if \eqref{eq:boundedness} is satisfied. According to Theorem \ref{t:boundedness} (2), this proves (1).

The statements (2) and (3) are immediate consequences of (1), and (4) is left as a little exercise for the reader.
	\end{proof}

	\begin{rema}[Disk Systems]
	\label{r:boundedness}
The boundedness condition \eqref{eq:boundedness} is closely related to the `disk systems' discussed by K. Petersen in \cite[pp. 416, 424]{Petersen}: fix a \textit{single} $\theta \in \mathbb{S}=\{c\in \mathbb{C}:|c|=1\}$ and consider the closed, shift-invariant subset $\Sigma _{(c,\theta )}\subset \{-1,1\}^\mathbb{Z}$ consisting of all sequences $v=(v_k)$ satisfying that $\sup_{m,n\ge0}\, \bigl|\sum_{k=-m}^{n} v_k\theta ^k\bigr| \le c$ for some fixed $c>0$. This system can be thought of as a kind of random walk restricted to a disk with radius $c$. In \cite{Petersen}, the author studies dynamical properties of $\Sigma _{(c,\theta )}$ depending on the parameters $\theta $ and $c$, like soficity, positivity of entropy, and the relation between entropy and the logarithmic growth rate of periodic points. In one of the examples in \cite{Petersen}, $\theta $ is chosen as a root of the polynomial $f=u^4-u^3-u^2-u+1$ in Example \ref{e:nonexpansive} (1).

For every $r>0$, the closed set $\Sigma _r=\{v\in V_f:\bar{\xi }^*(v)\le r\}$ is a `generalized disk system' in the spirit of \cite{Petersen}, but (possibly) involving several different irrational rotations and a bigger alphabet.
	\end{rema}

\smallskip We return to the connection between the spaces $V_f$ and $W_f$. The map $\bar{\xi }^*\colon \ell ^\infty (\mathbb{Z},\linebreak[0]\mathbb{Z})\longrightarrow \ell ^*(\mathbb{Z},\mathbb{R})$ in Theorem \ref{t:boundedness} can obviously not be expected to be shift-equi\-variant. Indeed,\vspace{-2mm}
	\begin{equation}
	\label{eq:commutation}
\mathsf{d}(n,v)\coloneqq\bar{\sigma }^ n \circ \bar{\xi }^*(v)-\bar{\xi }^*\circ \bar{\sigma }^ n(v)=
	\begin{cases}
\hphantom{-}\sum_{j=0}^{n-1}v_j \bar{\sigma }^{n-j}w ^\circ&\textup{if}\enspace n>0,
	\\
\hphantom{-}0&\textup{if}\enspace n=0,
	\\
-\sum_{j=1}^{n}v_{-j}\bar{\sigma }^{j-n}w ^\circ &\textup{if}\enspace n<0,
	\end{cases}
	\end{equation}
for every $n \in \mathbb{Z}$ and $v \in \ell ^\infty (\mathbb{Z},\mathbb{Z})$ (cf. \eqref{eq:homoclinic2}), and the resulting map
	\begin{equation}
	\label{eq:d}
\mathsf{d}\colon \mathbb{Z}\times \ell ^\infty (\mathbb{Z},\mathbb{Z})\longrightarrow W_f ^\circ
	\end{equation}
is continuous in $v$ for every $n\in \mathbb{Z}$ and satisfies the cocycle equation
\begin{equation}
\label{eq:d-cocycle}
\mathsf{d}(m,\bar{\sigma }^ nv)+\bar{\sigma }^ m \mathsf{d}(n,v)=\mathsf{d}(m+n,v)
\end{equation}
for every $m,n \in \mathbb{Z}$ and $v \in \ell ^\infty (\mathbb{Z},\mathbb{Z})$.

We set
	\begin{equation}
	\label{eq:xi*}
\xi ^*=\rho \circ \bar{\xi }^*\colon \ell ^\infty (\mathbb{Z},\mathbb{Z})\longrightarrow \mathbb{T}^\mathbb{Z}.
	\end{equation}
According to Theorem \ref{t:boundedness} (1), $\rho \circ f(\bar{\sigma })\circ \bar{\xi }^*(v)=f(\alpha _f)\circ \rho \circ \bar{\xi }^*(v)=0$ for every $v\in \ell ^\infty (\mathbb{Z},\mathbb{Z})$, so that
	\begin{equation}
	\label{eq:xi*2}
\xi ^*(\ell ^\infty (\mathbb{Z},\mathbb{Z})) \subset X_f.
	\end{equation}

Equation \eqref{eq:commutation} shows that $\xi ^*$ is \emph{equivariant modulo $X_f ^\circ$}, and our next result implies that both $\xi ^*$ and its restriction $\xi ^*|_{V_f}$ to $V_f$ are \emph{surjective modulo $X_f ^\circ$}.

\smallskip We recall the notation \eqref{eq:Br} and set, for every $r\ge0$,
	\begin{equation}
	\label{eq:BrWf}
	\begin{gathered}
\bar{B}_r(W_f)=W_f \cap \bar{B}_r(\ell ^ \infty (\mathbb{Z},\mathbb{R})),\qquad \bar{B}_r(W_f^\circ)=W_f^\circ \cap \bar{B}_r(\ell ^ \infty (\mathbb{Z},\mathbb{R})),
	\\
\bar{B}_r(V_f)=V_f \cap \bar{B}_r(\ell ^ \infty (\mathbb{Z},\mathbb{R})).
	\end{gathered}
	\end{equation}

	\begin{theo}
	\label{t:2}
Let $Y_f=W_f\cap [0,1)^\mathbb{Z}$, denote by $\bar{Y}_f$ the closure of $Y_f$ in $W_f$ \textup{(}or, equivalently, in $\ell ^\infty (\mathbb{Z},\mathbb{R})$\textup{)}, and set $Z_f=f(\bar{\sigma })(Y_f)$ and $\bar{Z}_f = f(\bar{\sigma })(\bar{Y}_f)=\overline{f(\bar{\sigma })(Y_f)}$.
	\begin{enumerate}
	\item
$\bar{Z}_f$ is a closed, bounded, shift-invariant subset of $V_f$ without isolated points.
	\item
$\bar{\sigma }$ is topologically transitive on $\bar{Z}_f$,
	\item
$\xi ^*(\bar{Z}_f)+X_f ^\circ= \xi ^*(\bar{Z}_f) + \rho \bigl(\bar{B}_r(W_f^\circ)\bigr)= X_f$ for some $r>0$,
	\item
$h(\bar{\sigma }|_{\bar{Y}_f}) = h(\bar{\sigma }|_{\bar{Z}_f})=h(\alpha _f)$.
	\end{enumerate}
	\end{theo}

Theorem \ref{t:2} (3) motivates the following definition.

	\begin{defi}[\cite{LiS2}]
	\label{d:pseudocover}
A closed, bounded, shift-invariant subset $V \subset \ell ^\infty (\mathbb{Z},\mathbb{Z})$ is a \emph{pseudo-cover} of $X_f$ if
	\begin{equation}
	\label{eq:pseudocover}
\xi ^*(V)+X_f ^\circ=X_f.
	\end{equation}
If $Y$ has the additional property that $h(\bar{\sigma }|_Y) = h(\alpha _f)$ it will be called an \emph{equal entropy pseudo-cover} of $X_f$.
	\end{defi}

	\begin{exam}
	\label{e:pseudocover}
For every $L\ge1$ we set
	\begin{displaymath}
V_L =\{v\in \ell ^\infty (\mathbb{Z},\mathbb{Z}):0\le v_k <L\enspace \textup{for every}\enspace k\in \mathbb{Z}\}.
	\end{displaymath}
If $L$ is sufficiently large, then $V_L$ is a pseudocover of $X_f$: indeed, if $L> 2\|f\|_1$, then $V_L \supset \bar{Z}_f + \bar{v} \coloneqq \{v+\bar{v}:v\in \bar{Z}_f\}$ and $\xi ^*(V_L)\supset \xi ^*(\bar{Z}_f) + \xi ^*(\bar{v})=X_f$, where $\bar{v}= (\dots , \|f\|_1,\|f\|_1,\|f\|_1,\dots) \in V_f$ by Corollary \ref{c:boundedness} (3).
	\end{exam}

	\begin{proof}[Proof of Theorem \ref{t:2} \textup{(1) -- (3)}]
The restriction of $\rho $ to $Y_f$ is a continuous bijection from $Y_f$ onto $X_f$. Clearly, none of the spaces $Y_f$, $\bar{Y}_f$, $\bar{Z}_f$ have isolated points. Furthermore, since $(X_f,\alpha _f)$ is topologically transitive, the same is true for $(Y_f,\bar{\sigma })$, $(\bar{Y}_f,\bar{\sigma })$ and $(\bar{Z}_f,\bar{\sigma })$. This proves (1) and (2).

We turn to (3). If $x\in X_f$, and if $y\in Y_f$ satisfies that $\rho (y)=x$, then $v=f(\bar{\sigma })(y)\in Z_f$ and $\|v\|_\infty \le \|f\|_1$. If $c>0$ is the constant appearing in Theorem \ref{t:boundedness} (3), then $\|\bar{\xi }^*(v)\|_\infty \le c$. According to Theorem \ref{t:boundedness} (1), $w=\bar{\xi }^*\circ f(\bar{\sigma })(y) - y\in W_f^\circ$. Since $\|w\|_\infty \le 1+c$ we conclude that $\bar{\xi }^*(Z_f) + \bar{B}_{1+c}(W_f^\circ)\supset Y_f$ and $\xi ^*(Z_f) + \rho \bigl(\bar{B}_{1+c}(W_f^\circ)\bigr)=X_f$.
	\end{proof}

The proof of Theorem \ref{t:2} (4) will be given in the next section.

\section{The proof of Theorem \ref{t:2} (4)}

For the proof of Theorem \ref{t:2} (4) we need several intermediate results, including a crucial lemma (Lemma \ref{l:hanfeng}) whose proof was communicated to me by Hanfeng Li.

We recall the cocycle $\mathsf{d}\colon \mathbb{Z}\times V_f\longrightarrow W_f^\circ$ in \eqref{eq:d} and define a continuous map $\tau \colon V_f\times W_f^\circ\longrightarrow V_f\times W_f^\circ$ by
	\begin{equation}
	\label{eq:tau}
\tau (v,w) = (\bar{\sigma }v,\bar{\sigma }w+\mathsf{d}(1,v))
	\end{equation}
for every $(v,w)\in V_f\times W_f^\circ$. Next we define a continuous map $\bar{\zeta }\colon V_f\times W_f^\circ\longrightarrow W_f$ by setting
	\begin{equation}
	\label{eq:barzeta}
\bar{\zeta }(v,w) = \bar{\xi }^*(v) + w
	\end{equation}
for every $(v,w)\in V_f\times W_f^\circ $. An elementary calculation shows that $\bar{\zeta }$ is $(\tau ,\bar{\sigma })$-equivariant. Hence the map
	\begin{equation}
	\label{eq:zeta}
\zeta =\rho \circ \bar{\zeta }\colon V_f\times W_f^\circ \longrightarrow X_f
	\end{equation}
is continuous and $(\tau ,\alpha _f)$-equivariant.

	\begin{lemm}
	\label{l:entropy}
\textup{(1)} There exists a compact $\tau $-invariant set $\mathsf{C}\subset \bar{Z}_f\times W_f^\circ$ such that $\zeta (\mathsf{C})=X_f$ and $\pi _1(\mathsf{C})=\bar{Z}_f$, where $\pi _1\colon \bar{Z}_f\times W_f^\circ\longrightarrow \bar{Z}_f$ is the first coordinate projection.

\textup{(2)} If $\mathsf{C}\subset \bar{Z}_f\times W_f^\circ$ is a compact $\tau $-invariant set such that $\pi _1(\mathsf{C})=\bar{Z}_f$, then $h(\tau |_{\mathsf{C}}) = h(\bar{\sigma }|_{\bar{Z}_f}) = h(\bar{\sigma }|_{\bar{Y}_f})\ge h(\alpha _f)$.
	\end{lemm}

	\begin{proof}
In the proof of Theorem \ref{t:boundedness} (3) we saw that there exists a constant $c>0$ such that $\|\bar{\xi }^*(z)\|_\infty \le c$ for every $z\in \bar{Z}_f$, and that $\bar{\zeta }(\bar{Z}_f\times \bar{B}_{1+c}(W_f^\circ))\supset \bar{Y}_f$. The boundedness of $\bar{\xi }^*$ also implies that $\sup_{n\in \mathbb{Z}}\,\sup_{z\in \bar{Z}_f}\|\mathsf{d}(n,z)\|_\infty \le 2c$ (cf. \eqref{eq:commutation}), so that
	\begin{displaymath}
\tau ^k(\bar{Z}_f\times \bar{B}_r(W_f^\circ))\subset \bar{Z}_f\times \bar{B}_{r+2c}(W_f^\circ)
	\end{displaymath}
for every $k\in \mathbb{Z}$. Hence there exists a closed $\tau $-invariant subset $\mathsf{C}\subset \bar{Z}_f\times \bar{B}_{r+2c}(W_f^\circ)$ containing $\bar{Z}_f\times \bar{B}_r(W_f^\circ)$ such that $\zeta (\mathsf{C})=X_f$ (cf. Theorem \ref{t:2} (3)). This proves (1).

For the proof of (2) we assume that $\mathsf{C}\subset \bar{Z}_f\times W_f^\circ$ is a compact $\tau $-invariant set satisfying that $\pi _1(\mathsf{C})=\bar{Z}_f$. Then $\mathsf{C}\subset \bar{Z}_f\times \bar{B}_r(W_f^\circ)$ for some $r>0$. We fix a metric $\vartheta _1$ on $\bar{Z}_f$, set $\vartheta _2(w,w')=\|w-w'\|_\infty $ for all $w,w'\in W_f^\circ$, and denote by $\vartheta ((z,w),(z',w'))=\vartheta _1(z,z') + \vartheta _2(w,w')$ the product metric on $\bar{Z}_f\times W_f^\circ$. Put, for every $N\ge1$, $z,z'\in \bar{Z}_f$, and $w,w'\in W_f^\circ$,
	\begin{gather*}
\vartheta _1^{(N)}(z,z')=\max_{k=0,\dots ,N-1}\vartheta _1 (\bar{\sigma }^kz,\bar{\sigma }^kz'),
	\\
\vartheta _2^{(N)}(w,w')=\max_{k=0,\dots ,N-1}\vartheta _2 (\bar{\sigma }^kw,\bar{\sigma }^kw') = \vartheta _2 (w,w'),
	\end{gather*}
where the last identity expresses the fact that $\bar{\sigma }$ acts isometrically on $W_f^\circ$, and set
	\begin{displaymath}
\vartheta ^{(N)}((z,w),(z',w')) = \vartheta _1^{(N)}(z,z\boldsymbol{}') + \vartheta _2^{(N)}(w,w').
	\end{displaymath}

Choose, for every $\varepsilon >0$, a minimal $(\vartheta _1^{(N)},\varepsilon /2)$-spanning set $S_1(N,\varepsilon /2)\subset \bar{Z}_f$ and a minimal $(\vartheta _2,\varepsilon /2)$-spanning set $S_2(\varepsilon /2)\subset \bar{B}_r(W_f^\circ)$. Then the set $S_1(N,\varepsilon /2)\linebreak[1]\times S_2(\varepsilon /2)$ $(\vartheta ^{(N)},\varepsilon )$-spans $\mathsf{C}$ (although it need not be contained in $\mathsf{C}$), and the definition of topological entropy in terms of spanning sets shows that
	\begin{align*}
h(\tau |_{\mathsf{C}}) &\le  \sup_{\varepsilon >0} \,\limsup_{N\to\infty } \,\tfrac{1}{N} \log |S_1(N,\varepsilon /2)\times S_2(\varepsilon /2)|
	\\
&= \sup_{\varepsilon >0} \,\limsup_{N\to\infty } \,\tfrac{1}{N} \log |S_1(N,\varepsilon /2)| = h(\bar{\sigma }|_{\bar{Z}_f}).
	\end{align*}
The reverse inequality $h(\tau |_{\mathsf{C}}) \ge h(\bar{\sigma }|_{\bar{Z}_f})$ follows from the fact that $\pi _1\colon \mathsf{C}\longrightarrow \bar{Z}_f$ is surjective and $(\tau ,\bar{\sigma })$-equivariant. Finally we observe that $h(\tau |_{\bar{Z}_f}) = h(\tau |_{\mathsf{C}})\linebreak[0]= h(\tau |_{\mathsf{C}}) \ge h(\alpha _f)$, where $\mathsf{C}$ is the set defined in (1), since the map $\zeta =\rho \circ \bar{\zeta }\colon \mathsf{C}\longrightarrow X_f$ is surjective and $(\tau ,\alpha _f)$-equivariant. This completes the proof of the lemma.
	\end{proof}

The final part of the proof of Theorem \ref{t:2} (4) will be to show that $h(\bar{\sigma }|_{\bar{Y}_f})=h(\alpha _f)$. We start with a definition.

	\begin{defi}
	\label{d:sauer-shelah}
Let $S$ be a finite nonempty set, and let $\mathcal{P}(S)$ denote the family of all subsets of $S$. A collection $\mathcal{F}\subset \mathcal{P}(S)$ \textit{shatters} a set $T\subset S$ if $\mathcal{P}(T)=\{F\cap T:F\in \mathcal{F}\}$. The set of all $T\in \mathcal{P}(S)$ shattered by $\mathcal{F}$ is denoted by $\textup{sh}(\mathcal{F})$.
	\end{defi}

	\begin{lemm}[Sauer-Shelah Lemma]
	\label{l:sauer-shelah}
Let $S$ be a finite set with $|S|=n\ge 1$ elements. If $k\in \{1,\dots ,n\}$, and if $\mathcal{F}\subset \mathcal{P}(S)$ is a collection of distinct sets with $|\mathcal{F}|>\sum_{i=0}^{k-1}\binom{n}{i}$, then $\mathcal{F}$ scatters a set $T\subset S$ of size $k$.
	\end{lemm}

	\begin{proof}
The following argument, taken from \cite[Theorem 1.1]{ARS}, proves a slightly stronger result due to A. Pajor \cite{Pajor}:

\smallskip{$(*)$}\enspace \textit{For every collection $\mathcal{F}\subset \mathcal{P}(S)$, $\textup{sh}(\mathcal{F})\ge |\mathcal{F}|$.}

\smallskip In order to prove $(*)$ we proceed by induction and assume that $|\mathcal{F}|=1$. Then $\textup{sh}(\mathcal{F})=\{\varnothing \}$, so that $|\textup{sh}(\mathcal{F})|=1=|\mathcal{F}|$.

For the induction step we assume that every collection $\mathcal{F}\subset \mathcal{P}(S)$ of size $\le k-1$ shatters at least $|\mathcal{F}|$ sets. Let $\mathcal{F}\subset \mathcal{P}(S)$ be a collection of $k$ distinct sets, and let $x\in S$ be an element of some, but not all, sets in $\mathcal{F}$. Put
	\begin{equation}
	\label{eq:F0F1}
\mathcal{F}_0=\{F\in \mathcal{F}:x\notin F\}, \qquad \mathcal{F}_1=\{F\smallsetminus \{x\}: x\in F\in \mathcal{F}\}.
	\end{equation}
By induction hypothesis, $|\textup{sh}(\mathcal{F}_i)|\ge |\mathcal{F}_i|$ for $i=0,1$. Then $|\mathcal{F}_0|+|\mathcal{F}_1|=|\mathcal{F}|=k$ and $\textup{sh}(\mathcal{F}_0) \cup \textup{sh}(\mathcal{F}_0)\subset \textup{sh}(\mathcal{F})$.

If a set $F\in \mathcal{P}(S)$ lies in $\textup{sh}(\mathcal{F}_0)\cap \textup{sh}(\mathcal{F}_1)$, then $\mathcal{P}(F)=\{E\cap F:E\in \mathcal{F}_0\}$, so that $x\notin F$. A glance at \eqref{eq:F0F1} shows that $F\cup\{x\}\in \textup{sh}(\mathcal{F})$. Hence $|\textup{sh}(\mathcal{F})| \ge |\textup{sh}(\mathcal{F}_0)\cup \textup{sh}(\mathcal{F}_1)| + |\textup{sh}(\mathcal{F}_0)\cap \textup{sh}(\mathcal{F}_1)| = |\textup{sh}(\mathcal{F}_0)| + |\textup{sh}(\mathcal{F}_1)|\ge k$, as claimed.

\smallskip Finally we note that the statement $(*)$ implies Lemma \ref{l:sauer-shelah}: if $|\textup{sh}(\mathcal{F})| \ge |\mathcal{F}|>\sum_{i=0}^{k-1}\binom{n}{i}$, then $\textup{sh}(\mathcal{F})$ must contain a set of size $\ge k$, since there are only $\sum_{i=0}^{k-1}\binom{n}{i}$ sets of size $\le k-1$ in $\mathcal{P}(S)$.
	\end{proof}

	\begin{lemm}
	\label{l:affine}
Let $V$ be a finite-dimensional vector space over $\mathbb{R}$, and let $k>\textup{dim}\,V$. Let furthermore $\phi _1,\dots ,\phi _k$ be affine functions on $V$ and $b_1,\dots ,b_k\in \mathbb{R}$. Then there exist $a_1,\dots ,a_k\in \{0,1\}$ such that $\bigcap_{j=1}^k W_j(a_j)=\varnothing $, where
	\begin{displaymath}
W_j(a_j) =
	\begin{cases}
\{v\in V:\phi _j(v)<b_j\}&\textup{if}\enspace a_j=1,
	\\
\{v\in V:\phi _j(v)\ge b_j\}&\textup{if}\enspace a_j=0.
	\end{cases}
	\end{displaymath}
	\end{lemm}

	\begin{proof}[Proof \textup{(}Hanfeng Li, personal communication\textup{)}.]
We use induction on $\textup{dim}\,V$. If $\textup{dim}\,V\linebreak[0]=1$ our assertion is evident. Assume therefore that $n\ge2$, that the assertion has been proved for $\textup{dim}\,V<n$, and that $\textup{dim}\,V=n$. If $\phi _k$ is constant, then either $W_k(0)$ or $W_k(1)$ is empty, and we can take the corresponding value of $a_k$ and choose $a_1,\dots ,a_{k-1}$ arbitrarily.

If $\phi _k$ is not constant, the hyperplane $X=\{v\in V:\phi _k(v)=b_k\}$ has dimension $n-1$, and the restrictions $\phi _j |_X$ are affine functions on $X$ for $j=1,\dots ,k-1$. According to our induction hypothesis we can find $a_1,\dots ,a_{k-1}\in\{0,1\}$ such that $\bigcap_{j=1}^{k-1} W_j(a_j)'=\varnothing $, where $W_j(a_j)'=W_j(a_j)\cap X$.

We claim that at least one of the sets $W_k(0) \cap \bigcap_{j=1}^{k-1} W_j(a_j)$ and $W_k(1) \cap \bigcap_{j=1}^{k-1} W_j(a_j)$ is empty. Indeed, if both are nonempty, and if $y_i\in W_k(i) \cap \bigcap_{j=1}^{k-1} W_j(a_j)$ for $i=0,1$, then some convex combination $x$ of $y_0$ and $y_1$ lies in $X$. Since $\bigcap_{j=1}^{k-1} W_j(a_j)$ is convex, $x\in X\cap \bigcap_{j=1}^{k-1} W_j(a_j) = \varnothing $, which is impossible. This contradiction completes the proof both of the induction step and of the lemma.
	\end{proof}

	\begin{lemm}
	\label{l:hanfeng}
$h(\bar{\sigma }|_{\bar{Y}_f})=h(\alpha _f)$.
	\end{lemm}

	\begin{proof}[Proof \textup{(}Hanfeng Li, personal communication\textup{)}.]
The map $\rho |_{Y_f}\colon Y_f\longrightarrow X_f$ is a continuous, shift-equivariant bijection. If $\mu $ is the probability measure on $Y_f$ satisfying $\mu (B)=\lambda _{X_f}(\rho (B))$ for every Borel set $B\subset Y_f$, then $h(\bar{\sigma }|_{\bar{Y}_f}) \ge h_\mu (\bar{\sigma })=h(\alpha _f)$, where the middle term denotes measure-theoretic entropy.

We assume that $h(\bar{\sigma }|_{\bar{Y}_f}) > h(\alpha _f)$ and will show that this leads to a contradiction. Define, for every finite subset $F\subset \mathbb{Z}$, pseudometrics $\vartheta _{\bar{Y}_f}^{(F)}$ and $\vartheta _{X_f}^{(F)}$ on $\bar{Y}_f$ and $X_f$ by setting
	\begin{equation}
	\label{eq:metrics}
	\begin{gathered}
\vartheta _{\bar{Y}_f}^{(F)}(x,x')=\max\nolimits_{k\in F}\,|x_k-x'_k| ,\enspace x,x'\in \bar{Y}_f,
	\\
\vartheta _{X_f}^{(F)}(y,y')= \max\nolimits_{k\in F}\,\dT y_k-y_k'\dT, \enspace y,y'\in X_f,
	\end{gathered}
	\end{equation}
where $\dT s-t\dT = \min_{k\in \mathbb{Z}}\,|a-b+k|$ denotes the usual distance between two elements $s=a+\mathbb{Z}$, $t=b+\mathbb{Z}$, of $\mathbb{T}=\mathbb{R}/\mathbb{Z}$.

For every $\varepsilon >0$ and every finite subset $F\subset \mathbb{Z}$ we choose a maximal \smash{$(\vartheta _{\bar{Y}_f}^{(F)},\varepsilon )$}-separated subset $D(\bar{Y}_f,F,\varepsilon )\subset \bar{Y}_f$ and a maximal $(\vartheta _{X_f}^{(F)},\varepsilon )$-separated subset $D(X_f,\linebreak[0]F,\varepsilon )\subset X_f$. Since $Y_f$ is dense in $\bar{Y}_f$ we may assume without loss in generality that $D(\bar{Y}_f,F,\varepsilon )\subset Y_f$.

If we set $F_N=\{0,\dots ,N-1\}$ for every $N\ge1$, then the definition of topological entropy implies that
	\begin{equation}
	\label{eq:top-ent}
	\begin{gathered}
h(\bar{\sigma }|_{\bar{Y}_f}) = \textstyle{\sup_{\varepsilon >0}\,\limsup_{N\to\infty }\frac{1}{N}\log |D(\bar{Y}_f,F_N,\varepsilon )|,}
	\\
h(\alpha _f) = \textstyle{\sup_{\varepsilon >0}\,\limsup_{N\to\infty }\frac{1}{N}\log |D(X_f,F_N,\varepsilon )|}.
	\end{gathered}
	\end{equation}

Since $h(\bar{\sigma }|_{\bar{Y}_f}) > h(\alpha _f)$ we can find $\varepsilon $ with $0<\varepsilon <1/4\|f\|_1$, $c>0$, and an increasing sequence $(N_k)_{k\ge1}$ of natural numbers, such that
	\begin{displaymath}
|D(\bar{Y}_f,F_{N_k},\varepsilon )| \ge |D(X_f,F_{N_k},\varepsilon )| \cdot e^{cN_k}
	\end{displaymath}
for every $k\ge1$.

We fix $k\ge1$ for the moment. Since the set $D(X_f,F_{N_k},\varepsilon )$ is maximal $(\vartheta _{X_f}^{(F_{N_k})},\varepsilon )$-separated,
	\begin{displaymath}
\smash[t]{X_f=\bigcup\nolimits _{x\in D(X_f,F_{N_k},\varepsilon )}\bar{B}_{\vartheta _{X_f}^{(F_{N_k})}}(x,\varepsilon /2),}
	\end{displaymath}
where $\bar{B}_{\vartheta _{X_f}^{(F_{N_k})}}(x,\varepsilon /2)$ denotes the closed $\vartheta _{X_f}^{(F_{N_k})}$-ball with centre $x$ and radius $\varepsilon /2$. We can thus choose an element $z\in D(X_f,F_{N_k},\varepsilon )$ such that the set
	\begin{equation}
	\label{eq:W1}
W_{N_k} \coloneqq \{y\in D(\bar{Y}_f,F_{N_k},\varepsilon ):\vartheta _{X_f}^{(F_{N_k})}(\rho (y),z)\le \varepsilon /2\}
	\end{equation}
has cardinality $\ge e^{cN_k}$. Let $\bar{z}\in Y_f$ be the unique element satisfying that $\rho (\bar{z})=z$ and define, for every $y\in W_{N_k}$, elements $\bar{y}\in [-1/4,5/4]^\mathbb{Z}$ and $y^*\in\{0,1\}^\mathbb{Z}$ by demanding that
	\begin{gather*}
|\bar{y}_n-\bar{z}_n|\le \varepsilon /2\quad \textup{and}\quad  \bar{y}_n= y_n\enspace (\textup{mod}\;1)\quad\textup{if}\quad n\in F_{N_k},\\
\bar{y}_n=y_n\quad\textup{if}\quad n\in \mathbb{Z}\smallsetminus F_{N_k},
	\end{gather*}
and setting
	\begin{displaymath}
y_n^*=|y_n-\bar{y}_n|\enspace \textup{for every}\enspace n\in \mathbb{Z}.
	\end{displaymath}
Since $\bar{y},\bar{z}\in W_f$, we obtain that $f(\bar{\sigma })(\bar{y}),f(\bar{\sigma })(\bar{z})\in \ell ^\infty (\mathbb{Z},\mathbb{Z})$, and the smallness of $|\bar{y}_n - \bar{z}_n|$ for $n\in F_{N_k}$ guarantees that
	\begin{equation}
	\label{eq:zeros}
[f(\bar{\sigma })(\bar{y}-\bar{z})]_n=0\enspace \textup{for every}\enspace n\in F_{N_k-m},
	\end{equation}
where $m = \textup{deg}(f)$ is the degree of $f$ (cf. \eqref{eq:f}).

By looking at the definition of $W_{N_k}$ we see that $y_n^*=1$ if $n\in F_{N_k}$ and $y_n$ and $\bar{z}_n$ lie close to opposite ends of the interval $[0,1]$, and $y^*_n=0$ otherwise (i.e., if either $n\in F_{N_k}$ and $|y_n-\bar{z}_n|\le \varepsilon /2$, or if $n\notin F_{N_k}$). Note that the map $y\mapsto y^*$ from $W_{N_k}$ to $\{0,1\}^{F_{N_k}}$ is injective: if $x,y\in W_{N_k}$ satisfy that $x^*=y^*$, then $|x_n - y_n|\le \varepsilon $ for every $n\in F_{N_k}$. Since $W_{N_k}$ is $(\vartheta _{\bar{Y}_f}^{(F_{N_k})},\varepsilon )$-separated this guarantees that $x=y$.

For every $y\in W_{N_k}$ and $n\in F_{N_k}$, exactly one of the following conditions is satisfied.
	\begin{enumerate}
	\item
[(a)] $\bar{y}_n-\bar{z}_n < -\bar{z}_n$, $y_n^*=1$, and $\bar{y}_n<0$, $\bar{z}_n<\varepsilon /2$,
	\item
[(b)] $\bar{y}_n-\bar{z}_n \ge -\bar{z}_n$, $y_n^*=0$, and $0\le \bar{y}_n<1$, $0\le \bar{z}_n <1$,
	\item
[(c)] $\bar{y}_n - \bar{z}_n \ge 1-\bar{z}_n$, $y_n^*=1$, and $\bar{y}_n \ge 1$, $\bar{z}_n \ge \varepsilon /2$,
	\item
[(d)] $\bar{y}_n - \bar{z}_n < 1-\bar{z}_n$, $y_n^* = 0$, and $0 \le \bar{y}_n<1$, $0\le \bar{z}_n <1$.
	\end{enumerate}
For every $n\in F_{N_k}$ we define a linear functional $\phi _n\colon \ell ^\infty (\mathbb{Z},\mathbb{R})\longrightarrow \mathbb{R}$ by setting
	\begin{equation}
	\label{eq:phi_n}
\phi _n(v)=
	\begin{cases}
\hphantom{-}v_n &\textup{if}\enspace \bar{z}_n <\varepsilon /2,
	\\
-v_n &\textup{if}\enspace \bar{z}_n \ge \varepsilon /2,
	\end{cases}
	\end{equation}
for every $v\in \ell ^\infty (\mathbb{Z},\mathbb{R})$. Then the following holds for every $y\in W_{N_k}$ and $n\in F_{N_k}$:
	\begin{equation}
	\label{eq:inequalities}
y_n^*=
	\begin{cases}
1\quad \textup{and}\quad
\phi _n(\bar{y}-\bar{z}) <
	\begin{cases}
-\bar{z}_n&\textup{if}\enspace \bar{z}_n<\varepsilon /2,
	\\
\bar{z}_n-1&\textup{if}\enspace \bar{z}_n\ge\varepsilon /2
	\end{cases}
	\\
0\quad \textup{and}\quad
\phi _n(\bar{y}-\bar{z}) \ge
	\begin{cases}
-\bar{z}_n&\textup{if}\enspace \bar{z}_n<\varepsilon /2,
	\\
\bar{z}_n-1&\textup{if}\enspace \bar{z}_n\ge\varepsilon /2
	\end{cases}
	\end{cases}
	\end{equation}

\smallskip We set $\overline{W}\negthinspace_{N_k}-\bar{z} = \{\bar{y}-\bar{z}: y\in W_{N_k}\}$, $W_{N_k}^* = \{y^*: y\in W_{N_k}\}$, and put
	\begin{align*}
V_{N_k}=\{v\in \ell ^\infty (\mathbb{Z},\mathbb{R})&: v_n=0\enspace \textup{for all}\enspace n\in \mathbb{Z}\smallsetminus F_{N_k}\enspace
	\\
&\textup{and}\enspace f(\bar{\sigma })(v)_n=0 \enspace \textup{for all}\enspace n\in F_{N_k-m}\}.
	\end{align*}
Since restriction of the map $f(\bar{\sigma })$ to the $N_k$-dimensional linear space $\{v\in\ell ^\infty (\mathbb{Z},\mathbb{R}):v_n=0\enspace \textup{for every}\enspace n\in \mathbb{Z}\smallsetminus F_{N_k}\}$ is injective, the space $V_{N_k}$ has dimension $m$. Furthermore, $\overline{W}\negthinspace_{N_k}-\bar{z}\subset V_{N_k}$ by \eqref{eq:zeros}.

We view every $y^*\in W_{N_k}^*$ as the set $\{n\in F_{N_k}:y^*_n=1\}$. As we just saw, the map $y\mapsto y^*$ is injective, so that $|W_{N_k}^*|\ge e^{cN_k}$. If $k$ is sufficiently large so that
	\begin{displaymath}
\sum\nolimits_{i=0}^{m+1}\textstyle{\binom{N_k}{i}} \le 1+N_k^{m+1} < e^{cN_k},
	\end{displaymath}
where $m=\textup{deg}(f)$, then Lemma \ref{l:sauer-shelah} guarantees that the family $W_{N_k}^*\subset \mathcal{P}(F_{N_k})$ shatters a set $T\subset F_{N_k}$ of size $m+1$. In other words, there exists, for every $(a_n)_{n\in T}\in \{0,1\}^T$, an element $y\in W_{N_k}$ with $y^*_n=a_n$ for every $n\in T$. This element $y$ satisfies that $\bar{y}-\bar{z}\in V_{N_k}$, and \eqref{eq:inequalities} shows that
	\begin{align*}
\phi _n(\bar{y}-\bar{z}) & < b_n \quad\textup{if}\enspace a_n = 1,
	\\
\phi _n(\bar{y}-\bar{z}) & \ge b_n \quad\textup{if}\enspace a_n = 0,
	\end{align*}
where\vspace{-2mm}
	\begin{displaymath}
b_n=
	\begin{cases}
-\bar{z}_n&\textup{if}\enspace \bar{z}_n<\varepsilon /2,
	\\
\bar{z}_n-1&\textup{if}\enspace \bar{z}_n\ge \varepsilon / 2.
	\end{cases}
	\end{displaymath}

By comparing this with Lemma \ref{l:affine} we see that $\textup{dim}\,V_{N_k}\ge m+1$, in violation of the fact that $\textup{dim}\,V_{N_k}=m$. Our assumption that $h(\bar{\sigma }|_{\bar{Y}_f}) > h(\alpha _f)$ has thus led to a contradiction. This proves the lemma.
	\end{proof}

	\begin{rema}
	\label{r:hanfeng}
The proof of Lemma \ref{l:hanfeng} by Hanfeng Li was formulated for principal algebraic actions of countable discrete amenable groups. The proof given here easily extends to that more general case.
	\end{rema}

	\begin{proof}[Proof of Theorem \ref{t:2} \textup{(4)}]
Lemma \ref{l:hanfeng} shows that $h(\bar{\sigma }|_{\bar{Y}_f})=h(\alpha _f)$. Since the map $f(\bar{\sigma })\colon \bar{Y}_f\longrightarrow \bar{Z}_f$ is surjective and shift-equivariant, $h(\alpha _f)=h(\bar{\sigma }|_{\bar{Y}_f})\ge h(\bar{\sigma }|_{\bar{Z}_f})\ge h(\alpha _f)$, where the last inequality follows from Lemma \ref{l:entropy} (2).
	\end{proof}

\section{Equal entropy covers of nonexpansive automorphisms}

From the nontriviality of the cocycle $\mathsf{d}$ in \eqref{eq:commutation} -- \eqref{eq:d-cocycle} it is clear that the map $\xi ^*\colon \bar{Z}_f\longrightarrow X_f$ is not $(\bar{\sigma },\alpha _f)$-equivariant. However, if we identify $\bar{Z}_f$ with the set $\bar{Z}_f\times \{0\}\subset \bar{Z}_f\times W_f^\circ$ and lift that set to a compact, $\tau $-invariant subset $\mathsf{C}\subset \bar{Z}_f\times W_f^\circ$ which projects onto the `base' $\bar{Z}_f$, we obtain an equal entropy cover $(\mathsf{C},\tau )_\zeta $ of $(X_f,\alpha _f)$ with covering map $\zeta \colon \mathsf{C}\longrightarrow X_f$ in \eqref{eq:zeta}.

	\begin{theo}
	\label{t:cover}
Let $\bar{Z}_f=f(\bar{\sigma })(\bar{Y}_f)\subset V_f$ be the compact shift-invariant set described in Theorem \ref{t:2}, and let $\tau \colon V_f\times W_f^\circ \longrightarrow V_f\times W_f^\circ $ be the homeomorphism defined in \eqref{eq:tau}. Then there exist a $\tau $-invariant compact set $\mathsf{C} \subset \bar{Z}_f\times W_f^\circ$ such that
	\begin{equation}
	\label{eq:coverC}
\pi _1(\mathsf{C}) = \bar{Z}_f.
	\end{equation}
Furthermore, if $\matheur{C} \subset \bar{Z}_f\times W_f^\circ$ is a compact, $\tau $-invariant set satisfying \eqref{eq:coverC}, and if $\zeta =\rho \circ \bar{\zeta }\colon V_f\times W_f^\circ\longrightarrow X_f$ is the $(\tau ,\alpha _f)$-equivariant map \eqref{eq:zeta}, then $h(\tau |_{\mathsf{C}})=h(\alpha _f)$ and $\zeta (\mathsf{C})=X_f$.
	\end{theo}

	\begin{proof}
By Theorem \ref{t:boundedness} (2), there exists a constant $c>0$ such that $\|\xi ^*(z)\|_\infty \le c$ for every $z\in \bar{Z}_f$. Hence $\|\mathsf{d}(n,z)\|_\infty \le 2c$ for every $z\in \bar{Z}_f$ and $n\in \mathbb{Z}$. It follows that there exists a compact, $\tau $-invariant subset $\mathsf{C}\subset \bar{Z}_f\times \bar{B}_2(W_f^\circ)$ with $\pi _1(\mathsf{C})=\bar{Z}_f$.

The set $V_f\times W_f^\circ $ is an additive group under component-wise addition, and we put, for every pair of sets $E\subset V_f\times W_f^\circ $, $F\subset W_f^\circ$, $T^F(E)=E+(\{0\}\times F)=\{(v,w+w'):(v,w)\in E\enspace \textup{and}\enspace w'\in F\}$. Note that $\zeta (T^F(E))=\zeta (E)+\rho (F)$, and that $T^{\bar{B}_s(W_f^\circ)}(\mathsf{C})$ is $\tau $-invariant for every $s\ge0$.

If $r>0$ is the constant appearing in Theorem \ref{t:2} (3), then
	\begin{displaymath}
\zeta (T^{\bar{B}_{r+2}(W_f^\circ)}(\mathsf{C})) \supset \xi ^*(\bar{Z}_f)+\rho (\bar{B}_r(W_f^\circ))=X_f.
	\end{displaymath}
Since $\zeta $ is $(\tau ,\alpha _f)$-equivariant and $T^{\bar{B}_s(W_f^\circ)}(\mathsf{C})$ is $\tau $-invariant, the set
	\begin{displaymath}
\zeta (T^{\bar{B}_s(W_f^\circ)}(\mathsf{C}) = \zeta (\mathsf{C}) +\rho (\bar{B}_s(W_f^\circ))\subset X_f
	\end{displaymath}
is closed and  $\alpha _f$-invariant for every $s>0$. Furthermore, since there exist finitely many elements $w^{(1)},\dots ,w^{(k)}$ in $W_f^\circ$ such that $\bigcup_{i=1}^k \bar{B}_s(W_f^\circ)+w^{(i)} \supset \bar{B}_{r+2}(W_f^\circ)$ and hence $\bigcup_{i=1}^kT^{\bar{B}_s(W_f^\circ) + w^{(i)}}(\mathsf{C}) \supset T^{\bar{B}_{r+2}(W_f^\circ)}(\mathsf{C})$, we obtain that
	$$
X_f=\bigcup_{i=1}^k\zeta (T^{\bar{B}_s(W_f^\circ)}(\mathsf{C})) + \rho (w^{(i)}).
	$$
It follows that the $\alpha _f$-invariant compact set $\zeta (T^{\bar{B}_s(W_f^\circ)}(\mathsf{C}))\subset X_f$ has positive Haar measure. Since $\alpha _f$ is ergodic we obtain that
	\begin{displaymath}
\zeta (T^{\bar{B}_s(W_f^\circ)}(\mathsf{C})) = X_f
	\end{displaymath}
for every $s>0$.

Fix $x\in X_f$. Then we can find, for every $m\ge1$, an element $(z^{(m)},w^{(m)})\in T^{\bar{B}_{1/m}(W_f^\circ)}(\mathsf{C})$ with $\zeta (z^{(m)},w^{(m)})=x$. The compactness of $T^{\bar{B}_2(W_f^\circ)}(\mathsf{C})$ allows us to find a convergent subsequence $((z^{(m_k)},w^{(m_k)}))_{k\ge1}$ of $((z^{(m)},w^{(m)}))_{m\ge1}$ with limit $(z,w)\in \mathsf{C}$, and the continuity of $\zeta $ implies that $\zeta (z,w)=x$. This proves that $\zeta (\mathsf{C})=X_f$. Furthermore, $h(\tau |_{\mathsf{C}}) = h(\bar{\sigma }|_{\bar{Z}_f})=h(\alpha _f)$ by the Lemmas \ref{l:entropy} and \ref{l:hanfeng}.
	\end{proof}

\section{Invariant sets and measures of nonexpansive automorphisms}

If $\alpha _f$ is expansive, the existence of Markov partitions or, more generally, of `nice' symbolic covers makes it a triviality to find infinite closed, $\alpha _f$-invariant subsets $C\subset X_f$. If $\alpha _f$ is nonexpansive, closed $\alpha _f$-invariant subsets have quite remarkable properties and are much more difficult to construct.

	\begin{theo}[\protect{\cite[Theorem 7.1]{LiS1}}]
	\label{t:LiS1-top}
Suppose that $\alpha _f$ is nonexpansive and totally irreducible.\footnote{\,\textit{Total irreducibility} of $\alpha _f$ means that $\alpha _f^k$ is irreducible for every $k\ge1$ or, equivalently, that the polynomial $g(u)=f(u^k)$ is irreducible for every $k\ge1$ (cf. \eqref{eq:f}).} Then any closed
$\alpha _f$-invariant subset $Y \subsetneq X_f$ intersects every central leaf $x+X_f^\circ$ of $\alpha _f$ in $X_f$ in a compact subset of that leaf.
	\end{theo}

Turning to $\alpha _f$-invariant probability measures on $X_f$, we note that the Haar measure $\lambda _{X_f}$ on $X_f$ is obviously invariant under $\alpha _f$, and that it is the unique $\alpha _f$-invariant measure of maximal entropy $h(\alpha _f) = \int_0^1\log|f(e^{2\pi is})|\,ds $ (cf. Theorem \ref{t:entropy}). What about other nonatomic and ergodic $\alpha _f$-invariant probability measures on $X_f$? If $\alpha _f$ is expansive, the existence of Markov partitions makes it easy to construct such measures, but if $\alpha _f$ is nonexpansive such measures are not so easy to come by. One reason for this is explained by the following theorem.

	\begin{theo}[\protect{\cite[Thoerem 5.1]{LiS1}}]
	\label{t:LiS1}
Suppose that $\alpha _f$ is nonexpansive and totally irreducible, and that $\mu $ is an $\alpha _f$-invariant probability measure on $X_f$ which is singular with respect to $\lambda _{X_f}$. Then the conditional measure $\rho _x$ on the central leaf $x+X_f^\circ$ through $x$ is finite for almost every $x\in X_f$.
	\end{theo}

	\begin{rema}
	\label{r:LiS1}
It is clear why we have to assume total irreducibility of $\alpha _f$: otherwise there exists an infinite closed subgroup $Y\subsetneq X_f$ whose orbit under $\alpha _f$ is finite, and by averaging the normalized Haar measure of $Y$ over the orbit of $Y$ we obtain an $\alpha _f$-invariant probability measure $\nu $ on $X_f$ which violates the conclusions of Theorem \ref{t:LiS1}.
	\end{rema}

	\begin{defi}[\protect{\cite[Definitions 1.1 and 1.2]{LiS1}}]
	\label{d:LiS1}
(1) An $\alpha _f$-invariant probability measure $\mu $ on $X_f$ is \emph{virtually hyperbolic} if there exists an $\alpha _f$-invariant Borel set $B\subset X_f$ with $\mu (B)=1$ which intersects every central leaf $x+X_f^\circ$ of $\alpha _f$ in $X_f$ in at most one point, i.e. with $B\cap(x+B)=\varnothing $ for every $x\in X^\circ$.

(2) Two $\alpha _f$-invariant probability measures $\mu_1,\mu_2$ on $X_f$ are \emph{centrally equivalent} if they have an invariant joining $\nu $ (i.e. an $(\alpha _f\times\alpha _f)$-invariant measure $\nu $ on $X_f\times X_f$ with marginals $\mu_1$ and $\mu_2$) so that, for $\nu\textsl{-a.e.}\;(x,y)\in X_f\times X_f$, $x$ and $y$ lie on the same central leaf. In other words,
	$$
x-y\in X_f^\circ\enspace \textup{for}\enspace \nu \textsl{-a.e.}\;(x,y)\in X_f\times X_f.
	$$
It is not difficult to check that centrally equivalent measures have the same entropy (Proposition \ref{p:central}).
	\end{defi}

	\begin{theo}[\protect{\cite[Theorem 1.3]{LiS1}}]
	\label{c:LiS1}
\textup{(1)} Suppose that $\alpha _f$ is nonexpansive and totally irreducible, and that $\mu $ is an $\alpha _f$-invariant and weakly mixing probability measure on $X_f$ which is singular with respect to $\lambda _{X_f}$. Then $\mu $ is virtually hyperbolic.

\textup{(2)} If an $\alpha _f$-invariant probability measure $\mu \ne \lambda _{X_f}$ is ergodic, but not necessarily weakly mixing, the it is centrally equivalent to a weakly mixing $\alpha _f$-invariant probability measure $\mu '$ on $X_f$. We write, for every $x\in X_f^\circ$, $m_x$ for the unique $\alpha _f$-invariant probability measure on $X_f^\circ$ --- and hence on $X_f$ --- concentrated on the compact orbit closure $\overline{\{\alpha _f^nx:n\in \mathbb{Z}\}}$ of $x$ under $\alpha _f$. Then $\mu $ is an ergodic component of $\mu ' * m_{x_0}$ for some $x_0\in X_f^\circ$.
	\end{theo}

Recently A. Quas and T. Soo proved a general result about the collection of all invariant probability measures of toral automorphisms which is all the more remarkable in view of Theorem \ref{t:LiS1} and Corollary \ref{c:LiS1}.

	\begin{theo}[\protect{\cite[Theorem 2]{quas}}]
	\label{t:quas}
Let $T$ be a homeomorphism of a compact metrizable space $Y$, and let $\mu $ be a $T$-invariant and ergodic probability measure on $Y$ with entropy $< h(\alpha _f)$. Then there exists a $(T,\alpha _f)$-equivariant Borel map $\phi \colon Y\longrightarrow X_f$ which is injective on a set $Y'\subset Y$ of full $\mu $-measure, and which sends $\mu $ to a fully supported $\alpha _f$-invariant probability measure $\nu =\phi _*\mu $ on $X_f$.

In other words we can find, for every ergodic system $(Y,T,\mu )$ whose entropy is less than that of $\alpha _f$, an $\alpha _f$-invariant and fully supported probability measure $\nu $ on $X_f$ such that $(X,T,\mu )$ and $(X_f,\alpha _f,\nu )$ are measurably conjugate.
	\end{theo}

The proof of Theorem \ref{t:quas} only uses certain general properties of $\alpha _f$: \textit{entropy expansiveness} (cf. Footnote \ref{entexp} \vpageref{entexp}), \textit{the small boundary property} (i.e., existence of arbitrarily small partitions of the space into Borel sets whose boundaries have measure zero w.r.t. every $\alpha _f$-invariant probability measure on $X_f$), and \textit{almost weak specification} (a specification property proved in \cite{marcus} to hold for ergodic toral automorphisms).

The remainder of this section is devoted to understanding the connection between $\alpha _f$-invariant measures on $X_f$ and shift-invariant probability measures on $\bar{Z}_f$. We start with some auxiliary results.

	\begin{prop}
	\label{p:product}
Let $Y_1,Y_2$ be compact metrizable spaces, and let $\tau _1,\tau _2$ be continuous $\mathbb{Z}^d$-actions on $Y_1$ and $Y_2$ such that the topological entropy $h(\tau _2)$ of $\tau _2$ is equal to zero. We write $\pi _i\colon Y_1\times Y_2\longrightarrow Y_i$ for the two coordinate projections. If $\mu $ is a $\tau _1\times \tau _2$-invariant probability measure on $Y_1\times Y_2$ we set $\mu _i=(\pi _i)_*(\mu )$. Then $h_\mu (\tau _1\times \tau _2)=h_{\mu _1}(\tau _1)$.
	\end{prop}

	\begin{proof}
Let $\mathcal{P}$ and $\mathcal{Q}$ be finite Borel partitions of $Y_1$ and $Y_2$, respectively, and set $\tilde{\mathcal{P}}=\{\tilde{P}=P \times Y_2:P\in \mathcal{P}\}$ and $\tilde{\mathcal{Q}}=\{\tilde{Q}=Y_1\times Q:Q\in \mathcal{Q}\}$. Then
	\begin{align*}
h_\mu (\tau _1\negthinspace \times \negthinspace \tau _2,\tilde{P}\vee \tilde{Q}) &= h_\mu (\tau _1\negthinspace \times \negthinspace \tau _2,\tilde{\mathcal{P}})
	\\
&+ \limsup_{n\to\infty }\frac1n H_\mu \Bigl(\bigvee\nolimits_{k=0}^{n-1} (\tau _1\negthinspace \times \negthinspace \tau _2)^{-k}\bigl(\tilde{\mathcal{Q}})\Bigm| \bigvee\nolimits_{k=0}^{n-1}(\tau _1\negthinspace \times \negthinspace \tau _2)^{-k}(\tilde{\mathcal{P}})\Bigr)
	\\
&\le h_{\mu _1}(\tau _1,\mathcal{P}) + \limsup_{n\to\infty }\tfrac1n H_{\mu _2}\Bigl(\bigvee\nolimits_{k=0}^{n-1} \tau _2^{-k}\bigl(Q)\Bigr)
	\\
&= h_{\mu _1}(\tau _1,\mathcal{P}) + h_{\mu _2}(\tau _2,\mathcal{Q}) = h_{\mu _1}(\tau _1,\mathcal{P})
	\end{align*}
by the variational principle \cite{Misiurewicz}. By varying $\mathcal{P}$ and $\mathcal{Q}$ we obtain that $h_\mu (\tau _1 \times \tau _2)\le h_{\mu _1}(\tau _1)$. The reverse inequality $h_{\mu _1}(\tau _1)\le h_\mu (\tau _1\times \tau _2)$ is obvious.
	\end{proof}

	\begin{coro}
	\label{c:product}
Let $\mu $ be an $\alpha _f\times \bar{\sigma }$-invariant probability measure on $X_f\times W_f^\circ$. If $\pi _1\colon X_f\times W_f^\circ\longrightarrow X_f$ and $\pi _2\colon X_f\times W_f^\circ\longrightarrow W_f^\circ$ are the coordinate projections and $\mu _i=(\pi _i)_*(\mu )$, then $h_\mu (\alpha _f\times \bar{\sigma })=h_{\mu _1}(\alpha _f)$.
	\end{coro}

	\begin{proof}
For every $k\ge1$, the set $X_f\times \bar{B}_k(W_f^\circ)\subset X_f\times W_f^\circ$ is compact and $(\alpha _f\times \bar{\sigma })$-invariant. We denote by $\mathbf{Q}^{(k)}$ the collection of all finite Borel partitions of $W_f^\circ$ containing the set $W_f^\circ\smallsetminus \bar{B}_k(W_f^\circ)$.

Let $\mathcal{P}$ be a finite Borel partition of $X_f$, and let $\mathcal{Q}\in \mathbf{Q}^{(k)}$. We set $\tilde{\mathcal{P}}=\{\tilde{P}=P\times W_f^\circ:P\in \mathcal{P}\}$ and $\tilde{\mathcal{Q}}=\{\tilde{Q}=X_f\times Q:Q\in \mathcal{Q}\}$. If $k$ is sufficiently large so that $\mu _2(\bar{B}_k(W_f^\circ))>0$, then \cite[Theorem 8.1]{Walters} implies that
	\begin{align*}
h_\mu (\alpha _f\negthinspace \times \negthinspace \bar{\sigma },\tilde{P}\vee \tilde{Q}) &\le h_{\mu _1}(\alpha _f,\mathcal{P}) + h_{\mu _2}(\bar{\sigma },\mathcal{Q})
	\\
&= h_{\mu _1}(\alpha _f,\mathcal{P}) + \mu _2(\bar{B}_k(W_f^\circ))h_{\mu _2^{(k)}}(\bar{\sigma },\mathcal{Q}'),
	\end{align*}
where $\mu _2^{(k)}=\frac{1}{\mu _2(\bar{B}_k(W_f^\circ))}\mu _2|_{\bar{B}_k(W_f^\circ)}$ is the normalized restriction of $\mu _2$ to the collection of Borel subsets of $\bar{B}_k(W_f^\circ)$, and $\mathcal{Q}'=\{Q\cap \bar{B}_k(W_f^\circ):Q\in \mathcal{Q}\}$ is the partition of $\bar{B}_k(W_f^\circ)$ induced by $\mathcal{Q}$. Hence
	\begin{displaymath}
h_\mu (\alpha _f \times \bar{\sigma },\tilde{P}\vee \tilde{Q}) \le h_{\mu _1}(\alpha _f) + h(\bar{\sigma }|_{\bar{B}_k(W_f^\circ)}) = h_{\mu _1}(\alpha _f),
	\end{displaymath}
since $\bar{\sigma }$ acts isometrically on $\bar{B}_k(W_f^\circ)$. By letting $k\to\infty $ and then varying $\mathcal{P}$ and $\mathcal{Q}$ we obtain that $h_\mu (\alpha _f\times \bar{\sigma })\le h_{\mu _1}(\alpha _f)$. The reverse inequality $h_{\mu _1}(\alpha _f)\le h_\mu (\alpha _f\times \bar{\sigma })$ is again obvious.
	\end{proof}

	\begin{coro}
	\label{c:product2}
For every shift-invariant probability measure $\mu $ on $\bar{Y}_f$, $h_\mu (\bar{\sigma })\linebreak[0]=h_{f(\bar{\sigma })_*\mu }(\bar{\sigma })$.
	\end{coro}

	\begin{proof}
We put $\nu =f(\bar{\sigma })_*\mu $. Since $f(\bar{\sigma })\colon \bar{Y}_f\longrightarrow \bar{Z}_f$ is shift-equivariant, $h_\nu (\bar{\sigma })\linebreak[0]\le h_\mu (\bar{\sigma })$. In order to prove the reverse inequality we note that the pre-image $f(\bar{\sigma })^{-1}(\{z\})\cap \bar{Y}_f$ of every $z\in \bar{Z}_f$ is closed and hence compact. By \cite[Section 1.5]{Partha} there exists a Borel map $\chi \colon \bar{Z}_f\longrightarrow \bar{Y}_f$ with $f(\bar{\sigma })\circ \chi (z)=z$ for every $z\in \bar{Z}_f$. For every $z\in \bar{Z}_f$ we set
	\begin{equation}
	\label{eq:cocycle}
c(z)=\chi \circ \bar{\sigma }(z)-\bar{\sigma }\circ \chi (z)\in W_f^\circ.
	\end{equation}
Consider the Borel map $\tau '\colon \bar{Z}_f\times W_f^\circ\longrightarrow \bar{Z}_f\times W_f^\circ$ given by
	\begin{displaymath}
\tau '(z,v) = (\bar{\sigma }z,\bar{\sigma }v-c(z))
	\end{displaymath}
for every $(z,v)\in \bar{Z}_f\times W_f^\circ$. Next we define an injective Borel map $\eta \colon \bar{Y}_f\longrightarrow \bar{Z}_f\times W_f^\circ$ by
	\begin{displaymath}
\eta (w)=(f(\bar{\sigma })(w),w-\chi \circ f(\bar{\sigma })(w))
	\end{displaymath}
and observe that $\eta \cdot \bar{\sigma }=\tau '\circ \eta $.

Due to the $(\bar{\sigma },\tau ')$-equivariance of $\eta $ the probability measure $\tilde{\nu }=\eta _*\mu $ on $\bar{Z}_f\times W_f^\circ$ is $\tau '$-invariant. Furthermore, $\tilde{\nu }$ is supported in the compact set $\bar{Z}_f\times \bar{B}_1(W_f^\circ)$, and $(\pi _1)_*\tilde{\nu }=\nu $. We decompose $\tilde{\nu }$ over $\bar{Z}_f$ by choosing a Borel measurable family $\nu _z,\,z\in \bar{Z}_f$, of probability measures on $W_f^\circ$ such that
	\begin{displaymath}
\int g(z,v)\,d\tilde{\nu }(z,v)=\int_{\bar{Z}_f}\int_{W_f^\circ}g(z,v)\, d\nu _z(v)\,d\nu (z)
	\end{displaymath}
for every bounded Borel map $g\colon \bar{Z}_f\times W_f^\circ\longrightarrow \mathbb{R}$. Since $\tilde{\nu }$ is $\tau '$-invariant, we obtain that
	\begin{equation}
	\label{eq:transform}
\int\int h(z,v)\,d\nu _z(v)\,d\nu (z)=\int\int h(\bar{\sigma }z,\bar{\sigma }v-c(z))\,d\nu _{z}(v)\,d\nu (z)
	\end{equation}
for every bounded Borel map $h\colon W_f^\circ\longrightarrow \mathbb{R}$ and $\nu \textit{-a.e.}\;z\in \bar{Z}_f$. We write $\pi _n\colon \ell ^\infty (\mathbb{Z},\mathbb{R})\linebreak[0]\longrightarrow \mathbb{R}$ for the $n$-th coordinate projection and set
	\begin{displaymath}
h_n=
	\begin{cases}
\pi _n&\textup{on}\enspace \bar{B}_1(W_f^\circ),
	\\
0&\textup{on}\enspace W_f^\circ\smallsetminus \bar{B}_1(W_f^\circ).
	\end{cases}
	\end{displaymath}
Then $b(z)_n\coloneqq \int \pi _n\,d\nu _z=\int h_n\,d\nu _z$ for every $z\in \bar{Z}_f$ and $n\in\mathbb{Z}$, and \eqref{eq:transform} (with $h=h_n$) shows that
	\begin{align*}
b(z)_n&=\int \pi _n(v) \,d\nu _z(v) = \int \pi _n(\bar{\sigma }v - c(\bar{\sigma }^{-1}z)) \,d\nu _{\bar{\sigma }^{-1}z}(v)
	\\
&= \int \pi _n(\bar{\sigma }v) \,d\nu _{\bar{\sigma }^{-1}z}(v) - c(\bar{\sigma }^{-1}z) = \int \pi _{n+1}(v)\, d\nu _{\bar{\sigma }^{-1}z}(v) - c(\bar{\sigma }^{-1}z)
	\end{align*}
for $\nu \textit{-a.e.}\; z\in \bar{Z}_f$. If we substitute $\bar{\sigma }z$ for $z$ in this equation we see that the map $z\mapsto b(z)$ from $\bar{Z}_f$ to $\bar{B}_1(K_f)$ satisfies that
	\begin{equation}
	\label{eq:coboundary}
c(z)=\bar{\sigma }\circ b(z)-b\circ \bar{\sigma }(z)\enspace \textup{for}\enspace \nu \textit{-a.e.}\;z\in \bar{Z}_f.
	\end{equation}

We replace the Borel map $\chi \colon \bar{Z}_f\longrightarrow \bar{Y}_f$ by $\chi '=\chi -b\colon \bar{Z}_f\longrightarrow W_f$ (which may, of course, no longer take values in $\bar{Y}_f$, but which is still bounded in norm) and define $\eta '\colon \bar{Y}_f\longrightarrow \bar{Z}_f\times W_f^\circ$ by
	\begin{displaymath}
\eta '(w)=(f(\bar{\sigma })(w),w-\chi '\circ f(\bar{\sigma })(w))
	\end{displaymath}
for every $w\in \bar{Y}_f$. Then $\eta '$ is injective and $\eta '\circ \bar{\sigma }=(\bar{\sigma }\times \bar{\sigma })\circ \eta '$.

Put $\tilde{\nu }'=\eta '_*\mu $. Then $(\pi _1)_*\tilde{\nu }'=\nu $, and by applying Proposition \ref{p:product} (or Corollary \ref{c:product}) we obtain that
	\begin{displaymath}
h_{\nu }(\bar{\sigma })=h_{\tilde{\nu }'}(\bar{\sigma }\times \bar{\sigma })=h_{\mu }(\bar{\sigma }).\qedhere
	\end{displaymath}
	\end{proof}

	\begin{prop}
	\label{p:central}
Let $\mu _1$ and $\mu _2$ be centrally equivalent $\alpha _f$-invariant probability measures on $X_f$. Then $h_{\mu _1}(\alpha _f)=h_{\mu _2}(\alpha _f)$.

If $\mu _1$ and $\mu _2$ are ergodic then there exist points $y_1,y_2\in X_f^\circ$ such that $\mu _i$ is an ergodic component of $\mu _j*m_{y_j}$ for $i,j\in \{1,2\}, i\ne j$ \textup{(}the measures $m_{y_i}$ are defined in Theorem \ref{t:LiS1} \textup{(2))}.
	\end{prop}

	\begin{proof}
Let $\Phi \colon X_f\times X_f\longrightarrow X_f\times X_f$ be the continuous $(\alpha _f\times \alpha _f)$-equivariant group isomorphism given by $\Phi (x,y)=(x,y-x)$. If $\nu $ is an $(\alpha _f\times \alpha _f)$-invariant joining of $\mu _1$ and $\mu _2$ such that $x-y\in X_f^\circ$ for $\nu \textsl{-a.e.}\;(x,y)\in X_f\times X_f$, then $\Phi _*(\nu )$ is supported on the $(\alpha _f\times \alpha _f)$-invariant Borel set $X_f\times X_f^\circ\subset X_f\times X_f$.

Since the map $\rho \colon W_f\longrightarrow X_f$ sends $W_f^\circ$ bijectively to $X_f^\circ$, the map $\tilde{\Phi }\colon X_f\times W_f^\circ\longrightarrow X_f\times X_f^\circ$, given by $\tilde{\Phi }(x,w)= (x,\rho (w))$, is a continuous $(\alpha _f\times \bar{\sigma },\alpha _f\times \alpha _f)$-equivariant bijection, and there exists a unique $(\alpha _f\times \bar{\sigma })$-invariant probability measure $\tilde{\nu }$ on $X_f\times W_f^\circ$ with $\tilde{\nu }(B)=\nu (\tilde{\Phi }(B))$ for every Borel set $B\subset X_f\times W_f^\circ$. Furthermore, if $\pi _1\colon X_f\times W_f^\circ\longrightarrow X_f$ is the first coordinate projection, then $\mu _1=(\pi _1)_*\tilde{\nu }$. By Corollary \ref{c:product}, $h_{\mu _1}(\alpha _f)=h_{\tilde{\nu }}(\alpha _f\times \bar{\sigma }) = h_\nu (\alpha _f\times \alpha _f)$. Similarly we see that $h_{\mu _2}(\alpha _f)=h_{\nu }(\alpha _f\times \alpha _f)$. This proves that $h_{\mu _1}(\alpha _f)=h_{\mu _2}(\alpha _f)$.

If the measures $\mu _i$ are ergodic we can find an $(\alpha _f\times \alpha _f)$-invariant and ergodic joining $\nu $ of $\mu _1$ and $\mu _2$. The $(\alpha _f\times \alpha _f)$-invariant probability measure $\Phi _*\nu $ on $X_f\times X_f^\circ$ described above is again ergodic. For every $y\in X_f^\circ$ we write $C_y\subset X_f^\circ$ for the orbit closure $\overline{\{\alpha _f^ky:k\in \mathbb{Z}\}}\subset X_f^\circ$ and $m_y$ for the unique $\alpha _f$-invariant probability measure on $C_y$. Since $X_f\times C_y$ is $(\alpha _f\times \alpha _f)$-invariant for every $y\in X_f^\circ$, only one of these sets can have positive --- and hence full --- measure w.r.t. $\Phi _*\nu $. By translating this back to our joining $\nu $ we see that $\mu _2$ is an ergodic component of $\mu _1*m_y$. Similarly we see that there exists a $y'\in X_f^\circ$ such that $\mu _1$ is an ergodic component of $\mu _1*m_{y'}$.
	\end{proof}

We can formulate the notion of central equivalence also for $\bar{\sigma }$-invariant probability measures on $W_f$: two $\bar{\sigma }$-invariant probability measures $\mu_1,\mu_2$ on $W_f$ are \emph{centrally equivalent} if they have a $(\bar{\sigma }\times\bar{\sigma })$-invariant joining $\nu $ on $W_f\times W_f$ so that, for $\nu\textsl{-a.e.}\;(w,w')\in W_f\times W_f$, $w-w'\in W_f^\circ$. The following result is proved exactly like Proposition \ref{p:central}.

	\begin{prop}
	\label{p:central2}
Let $\mu _1$ and $\mu _2$ be centrally equivalent $\bar{\sigma }$-invariant probability measures on $W_f$. Then $h_{\mu _1}(\bar{\sigma })=h_{\mu _2}(\bar{\sigma })$.

If $\mu _1$ and $\mu _2$ are ergodic then there exist points $y_1,y_2\in W_f^\circ$ such that $\mu _i$ is an ergodic component of $\mu _j*\bar{m}_{y_j}$ for $i,j\in \{1,2\},\, i\ne j$. Here $\bar{m}_{y_i}$ is the unique $\bar{\sigma }$-invariant probability measure on the orbit closure $\overline{\{\bar{\sigma }^ky_i:k\in \mathbb{Z}\}}\subset W_f^\circ$.
	\end{prop}

	\begin{prop}
	\label{p:pre-images}
Let $\nu $ be a $\bar{\sigma }$-invariant probability measure on $\bar{Z}_f$. Then there exists a $\bar{\sigma }$-invariant probability measure $\mu $ on $\bar{Y}_f$ with $\nu =f(\bar{\sigma })_*\mu $ and hence $h_\mu (\bar{\sigma })=h_\nu (\bar{\sigma })$. If $\nu $ is ergodic under $\bar{\sigma }$, then $\mu $ can also be chosen to be ergodic.

If $\mu '$ is a $\bar{\sigma }$-invariant probability measure on $W_f$ such that $f(\bar{\sigma })_*\mu '=\nu $, then $\mu '$ is centrally equivalent to $\mu $.
	\end{prop}

	\begin{proof}
The first assertion is clear from the compactness of $\bar{Y}_f$, the continuity of the shift-equivariant map $f(\bar{\sigma })\colon \bar{Y}_f\longrightarrow \bar{Z}_f$, and Proposition \ref{p:central2}.

In order to verify the central equivalence of $\mu $ and $\mu '$ we choose decompositions $\{\mu _z:z\in \bar{Z}_f\}$ and $\{\mu _z':z\in \bar{Z}_f\}$ of $\mu $ and $\mu '$ into probability measures such that $\mu _z(f(\bar{\sigma })^{-1}(\{z\})\cap \bar{Y}_f)= \mu _z'(f(\bar{\sigma })^{-1}(\{z\})\cap W_f) = 1$ for every $z\in \bar{Z}_f$, $\mu =\int \mu _z\,d\nu (z)$, $\mu '=\int \mu _z'\,d\nu (z)$, and the maps $z\mapsto \mu _z$ and $z\mapsto \mu _z'$ are Borel. For every $z\in \bar{Z}_f$ we define the product measure $\tilde{\mu }_z=\mu _z\times \mu _z'$ on $(f(\bar{\sigma })^{-1}(\{z\})\cap\bar{Y}_f)\times (f(\bar{\sigma })^{-1}(\{z\})\cap W_f)\subset \bar{Y}_f\times W_f$. The probability measure $\tilde{\mu }$ on $\bar{Y}_f\times W_f$, defined by
	\begin{displaymath}
\int g\, d\tilde{\mu }=\int g(y_1,y_2)\, d\tilde{\mu }_z(y_1,y_2)\,d\nu (z)
	\end{displaymath}
for every bounded real-valued Borel map $g$ on $\bar{Y}_f\times W_f$, is a joining of $\mu $ and $\mu '$ such that $y_1-y_2\in W_f^\circ$ for $\tilde{\mu }\textsl{-a.e.}\;(y_1,y_2)\in \bar{Y}_f\times W_f$.
	\end{proof}

We denote by $\bar{Z}_f'$ and $X_f'$ the set of doubly transitive points in $\bar{Z}_f$ and $X_f$, respectively, and put $Z_f'=Z_f\cap \bar{Z}_f'$ (for notation we refer to Theorem \ref{t:2}).

	\begin{lemm}
	\label{l:minimal}
Let $\mathsf{C}_f\subset \bar{Z}_f\times W_f^\circ$ be a compact $\tau $-invariant subset which is minimal with respect to the condition that $\pi _1(\mathsf{C}_f)=\bar{Z}_f$ \textup{(}cf. Theorem \ref{t:cover}\textup{)}. Then $(\mathsf{C}_f,\tau )$ is topologically transitive. Furthermore there exists a continuous map $b\colon \bar{Z}_f'\longrightarrow W_f^\circ$ such that the set of doubly transitive points in $\mathsf{C}_f$ is given by $\mathsf{C}_f'=\{(z,b(z)):z\in \bar{Z}_f'\}$, and
	\begin{equation}
	\label{eq:coboundary1}
\mathsf{d}(z)=b\circ \bar{\sigma }(z)-\bar{\sigma }\circ b(z)
	\end{equation}
for every $z\in \bar{Z}_f'$.
	\end{lemm}

	\begin{proof}
Let $R\colon w\mapsto R^w$ be the action of $W_f^\circ$ on $V_f\times W_f^\circ$ given by $R^w(v,w')=(v,w+w')$. Our minimality condition on $\mathsf{C}_f$ implies that $\mathsf{C}_f'\cap R^w(\mathsf{C}_f')=\varnothing $ for every nonzero $w\in W_f^\circ$. In other words, there exists a map $\mathsf{b}\colon \bar{Z}_f'\longrightarrow W_f^\circ$ such that $\mathsf{C}_f'=\{(z,b(z)):z\in \bar{Z}_f'\}$. Since $\mathsf{C}_f$ is closed, $b$ is continuous, and \eqref{eq:coboundary1} follows from the $\tau $-invariance of $\mathsf{C}_f$.
	\end{proof}

In view of Lemma \ref{l:minimal} we can define a continuous $(\bar{\sigma },\tau )$-equivariant bijection $\mathsf{b}\colon \mathsf{C}_f'\longrightarrow \bar{Z}_f'$ by setting
	\begin{displaymath}
\mathsf{b}(z)=(z,b(z))\enspace \textup{for every}\enspace z\in Z_f'.
	\end{displaymath}

	\begin{lemm}
	\label{l:injective}
Let $\bar{\zeta }^*=\bar{\zeta }\circ \mathsf{b} \colon \bar{Z}_f'\longrightarrow W_f$, and let $\zeta ^*=\rho \circ \bar{\zeta }^*\colon \bar{Z}_f'\longrightarrow X_f$. Then $\zeta ^*$ is $(\bar{\sigma },\alpha _f)$-equivariant.

\smallskip \textup{(1)} If $\nu $ is a fully supported $\bar{\sigma }$-invariant and ergodic probability measure on $\bar{Z}_f$, then $\bar{\mu }\coloneqq \bar{\zeta }^*_*\nu $ is a well-defined $\bar{\sigma }$-invariant and ergodic probability measure on $W_f$ with $h_{\bar{\mu }}(\bar{\sigma })=h_\nu (\bar{\sigma })$.

\smallskip \textup{(2)} If $\mu $ is a fully supported, $\bar{\sigma }$-invariant and ergodic probability measure on $\bar{Y}_f$, then $\nu =f(\bar{\sigma })_*\mu $ is fully supported, $\bar{\sigma }$-invariant and ergodic on $\bar{Z}_f$, and the shift-invariant probability measure $\bar{\mu }\coloneqq \bar{\zeta }^*_*\nu $ on $W_f$ is centrally equivalent to $\mu $.
	\end{lemm}

	\begin{proof}
The $(\bar{\sigma },\alpha _f)$-equivariance of $\zeta ^*$ follows from the equivariance of $\mathsf{b}$. If $\nu $ is a $\bar{\sigma }$-invariant, ergodic, and fully supported probability measure on $\bar{Z}_f$, then $\nu (\bar{Z}_f')=1$, and $\bar{\mu }=\bar{\zeta }^*_*\nu $ is well-defined. Furthermore, since $\bar{\zeta }^*$ is injective on $\bar{Z}_f'$, $h_{\bar{\mu }}(\bar{\sigma })=h_\nu (\bar{\sigma })$.
	\end{proof}

The following theorem summarizes the connection between shift-invariant probability measures on $\bar{Z}_f$, $\bar{Y}_f$ and $X_f$.

	\begin{theo}
	\label{t:measures}
\textup{(1)} Let $\nu $ be a $\bar{\sigma }$-invariant probability measure on $\bar{Z}_f$. Then there exists a $\bar{\sigma }$-invariant probability measure $\mu $ on $\bar{Y}_f$ such that $f(\bar{\sigma })_*\mu =\nu $ and hence $h_\mu (\bar{\sigma })=h_\nu (\bar{\sigma })$. The measure $\mu $ is unique up to central equivalence. If $\nu $ is ergodic under $\bar{\sigma }$, then $\mu $ can also be chosen to be $\bar{\sigma }$-ergodic.

Furthermore, the probability measure $\rho _*\mu $ on $X_f$ is $\alpha _f$-invariant, but may have lower entropy than $\nu $.

\smallskip \textup{(2)} If the probability measure $\nu $ in \textup{(1)} is fully supported and ergodic on $\bar{Z}_f$, then $\nu (\bar{Z}_f')=1$, and the probability measure $\bar{\zeta }_*^*\nu $ on $W_f$ is well-defined, $\bar{\sigma }$-invariant, ergodic, and centrally equivalent to the measure $\mu $ in \textup{(1)}.

\smallskip \textup{(3)} Suppose that $\nu $ is a fully supported and ergodic $\bar{\sigma }$-invariant probability measure on $\bar{Z}_f$ with $\nu (Z_f)=1$. Then the $\bar{\sigma }$-invariant probability measure $\mu $ on $\bar{Y}_f$ in \textup{(1)} with $f(\bar{\sigma })_*\mu =\nu $ satisfies that $\mu (Y_f)=1$ and hence $h_\nu (\bar{\sigma })=h_\mu (\bar{\sigma })=h_{\rho _*\mu }(\alpha _f)$. Moreover, if $\bar{\mu }$ is a $\bar{\sigma }$-invariant and ergodic probability measure on $W_f$ with $f(\bar{\sigma })_*\bar{\mu }=\nu $, then $h_\nu (\bar{\sigma })=h_{\bar{\mu }}(\bar{\sigma })=h_{\rho _*\bar{\mu }}(\alpha _f)=h_{\rho _*\mu }(\alpha _f)$.

\smallskip \textup{(4)} Finally, let $\mu '$ be a fully supported $\alpha _f$-invariant and ergodic probability measure on $X_f$, and let $\mu $ be the unique probability measure on $Y_f$ such that $\rho _*\mu =\mu '$. Put $\nu =f(\bar{\sigma })_*\mu $ and $\mu ''=\zeta ^*_*\nu $. Then $\mu '$ and $\mu ''$ are centrally equivalent.
	\end{theo}

	\begin{proof}
Assertion (1) was proved in Proposition \ref{p:pre-images}. If $\nu $ is a fully supported and ergodic $\bar{\sigma }$-invariant probability measure on $\bar{Z}_f$, then $\nu (\bar{Z}_f')=1$, and Lemma \ref{l:injective} shows that $\bar{\zeta }_*^*\nu $ is well-defined and ergodic on $\bar{Y}_f$. The central equivalence of $\mu $ and $\bar{\zeta }_*^*\nu $ was verified in Lemma \ref{l:injective} and Proposition \ref{p:pre-images}. This proves (2).

We turn to (3). Let $\nu $ be a fully supported and ergodic $\bar{\sigma }$-invariant probability measure on $\bar{Z}_f$ with $\nu (Z_f)=1$. Since $f(\bar{\sigma })^{-1}(Z_f)\cap \bar{Y}_f\subset Y_f$, any shift-invariant measure $\mu $ on $\bar{Y}_f$ satisfies that $\mu (Y_f)=1$. Hence $h_{\rho _*\mu }(\alpha _f)=h_\mu (\bar{\sigma })=h_\nu (\bar{\sigma })$.

If $\bar{\mu }$ is a $\bar{\sigma }$-invariant and ergodic probability measure on $W_f$ with $f(\bar{\sigma })_*\bar{\mu }=\nu $, then $\bar{\mu }$ is centrally equivalent to $\mu $ by Proposition \ref{p:pre-images}, and hence, by Proposition \ref{p:central2}, an ergodic component of $\mu *\bar{m}_y$  for some $y\in W_f^\circ$. Then $\rho _*\bar{\mu }$ is an ergodic component of $\rho _*\mu *m_y$ (cf. Proposition \ref{p:central}), and hence centrally equivalent to $\rho _*\mu $. It follows that $h_{\rho _*\bar{\mu }}(\alpha _f) = h_{\rho _*\mu }(\alpha _f)=h_\nu (\bar{\sigma })$.

Since (4) follows from (3), the theorem is proved completely.
	\end{proof}

	\begin{exam}
	\label{e:Lebesgue}
If we set $\mu '=\lambda _{X_f}$ in Theorem \ref{t:measures} (4) we see that $\mu ''=\mu '$, since $h_{\mu '}(\alpha _f)=h_{\mu ''}(\alpha _f)=h(\alpha _f)$ and $\mu '$ is the unique $\alpha _f$-invariant measure of maximal entropy on $X_f$. This implies that $\zeta ^*(\bar{Z}_f') = X_f\;(\textup{mod}\,\lambda _{X_f})$.
	\end{exam}

\section{Some open problems}

\subsection{The space $\bar{Z}_f$}

The space $\bar{Z}_f$ defined in Theorem \ref{t:2} is not canonical. We could have defined, for any $c\in \mathbb{R}$, $Y_f^{(c)}=W_f\cap [c,c+1)^\mathbb{Z}$, put $Z_f^{(c)}=f(\bar{\sigma })(Y_f^{(c)})$, and written $\bar{Y}_f^{(c)}$ and $\bar{Z}_f^{(c)}$ for the corresponding closures. As far as I can tell this would not have made a significant difference to any of the properties of these sets (although the actual sets would have changed, of course).

The connection between $V_f$ and, by implication, $\bar{Z}_f$, and the `disk systems' in \cite{Petersen} was mentioned in Remark \ref{r:boundedness}. In \cite{Petersen}, the author attributes to B. Marcus the conjecture that certain disk systems discussed there are almost sofic\footnote{A shift space $\Omega \subset \mathsf{A}^\mathbb{Z}$ with finite alphabet $\mathsf{A}$ is \textit{almost sofic} if there exists, for every $\varepsilon >0$, a \textit{SFT} $\Omega '\subset \Omega $ with entropy $h(\Omega ')>h(\Omega )-\varepsilon $.} and raises the question whether the system associated with the Salem polynomial in Example \ref{e:nonexpansive} (1) is almost sofic. I am tempted towards the following conjecture.

	\begin{conj}
	\label{c:sofic}
If $f$ is irreducible and nonhyperbolic, the shift space $\bar{Z}_f$ is not almost sofic.
	\end{conj}

Being almost sofic is a useful property for the purpose of data encoding (cf. \cite[p. 419]{Petersen}). Although the spaces $\bar{Z}_f$ may not share this property, they are still well-behaved in other respects. For example, their entropies are equal to the logarithmic growth rates of their periodic points by Theorem \ref{t:periodic} (in \cite{Petersen}, the author calls such shift spaces `periodically saturated').

\medskip Let mention another question about $\bar{Z}_f$.

	\begin{prob}
	\label{p:intrinsic}
If $f$ is irreducible and nonhyperbolic, is $(\bar{Z}_f,\bar{\sigma })$ intrinsically ergodic in the sense of \cite{Weiss} (i.e., is there a unique shift-invariant measure of maximal entropy)?
	\end{prob}

\subsection{Pseudocovers}

Pseudo-covers come in two flavours. According to Definition \ref{d:pseudocover}, a closed, shift-invariant subset $V\subset \ell ^\infty (\mathbb{Z},\mathbb{Z})$ is a pseudo-cover if $\xi ^*(V) + X_f^\circ = X_f$. However, the pseudo-cover $\bar{Z}_f$ in Theorem \ref{t:2} has the much stronger property that there exists a compact subset $K\subset X^\circ$ such that $\xi ^*(V) + K = X_f$. Let me call a pseudo-cover satisfying this stronger condition a \textit{strong pseudo-cover}. Strong pseudo-covers are much easier to handle than general pseudo-covers, since one can apply compactness arguments (as we did in Theorem \ref{t:cover} and Lemma \ref{l:minimal}). On the other hand, strong pseudo-covers appear to have quite a complicated structure. \textit{Can one find pseudo-covers which can be described more explicitly, e.g., sofic or almost sofic?}

\smallskip The following Problems \ref{p:pseudo} and \ref{p:funny} are aimed in this direction.

	\begin{prob}[Beta-shifts]
	\label{p:pseudo}
Let $f\in R_1$ be irreducible. For every $L\ge1$ we set
	\begin{equation}
	\label{eq:VL}
V_L =\{v\in \ell ^\infty (\mathbb{Z},\mathbb{Z}):0\le v_k <L\enspace \textup{for every}\enspace k\in \mathbb{Z}\},
	\end{equation}
and we put
	\begin{equation}
	\label{eq:VL*}
\smash{V_L^*=V_L\smallsetminus \bigcup\nolimits_{\{h\in \ell ^1(\mathbb{Z},\mathbb{Z}):h \succ 0\}}(V_L+f(\bar{\sigma })h)}
	\end{equation}
as in \eqref{eq:W*}.

If $f$ is hyperbolic, and if $L$ is sufficiently large so that $\xi (V_L)=X_f$, then $(V_L^*,\bar{\sigma })_\xi $ is an equal entropy symbolic cover of $(X_f,\alpha _f)$ which is actually sofic (Proposition \ref{p:W*}). The proof that $\xi (V_L^*)= X_f$ in \cite{S3} depends on a compactness argument which is not available if $f$ is nonhyperbolic.

If $f$ is nonhyperbolic, and if $L>2\|f\|_1$, then $V_L$ is still a pseudo-cover of $X_f$ by Example \ref{e:pseudocover}, but it is not at all clear how big the space $V_L^*$ is (cf. \cite{S4}).

If $f$ is a Salem polynomial (like the polynomial $f=u^4-u^3-u^2-u+1$ in Example \ref{e:nonexpansive} (1)), and if $\beta $ is the large root of $f$, then
	\begin{displaymath}
V_{L}^* \supset V_\beta
	\end{displaymath}
for every $L>\beta $, where $V_\beta $ is the beta-shift associated with $\beta $. In particular, $h(V_L^*)=h(V_\beta )= \log\,\beta =h(\alpha _f)$.
	\begin{itemize}\itemsep=0mm
	\item
Is the two-sided beta-shift $V_\beta $  a pseudo-cover of $X_f$?
	\item
In the special case where $f=u^4-u^3-u^2-u+1$, is the two-shift $\{0,1\}^\mathbb{Z}$ a pseudo-cover of $X_f$?
	\end{itemize}
These questions were part of the original motivation of the paper \cite{LiS2}.
	\end{prob}

	\begin{prob}
	\label{p:funny}
Let $f$ be one of the polynomials in Example \ref{e:nonexpansive} (2), e.g., $f=2u^2-u+2$. If $V_2=\{0,1\}^\mathbb{Z}$, then $V_2=V_2^*$ (cf. \eqref{eq:VL} -- \eqref{eq:VL*}) and $h(V_2) = \log2 = h(\alpha _f)$. \textit{Is $V_2$ a pseudocover of $X_f$?} For the polynomial $f=5u^2-6u+5$ in Example \ref{e:nonexpansive} the analogous question would be \textit{whether $V_5=V_5^*$ is an equal entropy pseudo-cover of $X_f$}.
	\end{prob}

	\end{document}